\documentclass[12pt]{amsart}

\usepackage[colorlinks = true,
            linkcolor = red,
            urlcolor  = blue,
            citecolor = blue,
            anchorcolor = blue]{hyperref}
\usepackage{color}
\numberwithin{equation}{section}

\newcommand{\lab}{\label}

\newcommand{\ben}{\begin{enumerate}}
\newcommand{\een}{\end{enumerate}}

\newcommand{\bea}{\begin{eqnarray}}
\newcommand{\ba}{\begin{array}}
\newcommand{\bean}{\begin{eqnarray*}}
\newcommand{\ea}{\end{array}}
\newcommand{\eea}{\end{eqnarray}}
\newcommand{\eean}{\end{eqnarray*}}
\newcommand{\beq}{\begin{equation}}
\newcommand{\eeq}{\end{equation}}
\newcommand{\bthm}{\begin{thm}}
\newcommand{\ethm}{\end{thm}}
\newcommand{\blem}{\begin{lem}}
\newcommand{\elem}{\end{lem}}
\newcommand{\bprop}{\begin{prop}}
\newcommand{\eprop}{\end{prop}}
\newcommand{\bcor}{\begin{cor}}
\newcommand{\ecor}{\end{cor}}
\newcommand{\bdfn}{\begin{dfn}}
\newcommand{\edfn}{\end{dfn}}
\newcommand{\brem}{\begin{rem}}
\newcommand{\erem}{\end{rem}}
\newcommand{\bpf}{\begin{proof}}
\newcommand{\epf}{\end{proof}}
\newcommand{\bfact}{\begin{fact}}
\newcommand{\efact}{\end{fact}}
\newcommand{\bobs}{\begin{obs}}
\newcommand{\eobs}{\end{obs}}

\alph{enumii} \roman{enumiii}

\newtheorem{thm}{Theorem}[section]
\newtheorem{prop}[thm]{Proposition}
\newtheorem{lem}[thm]{Lemma}

\newtheorem{cor}[thm]{Corollary}
\newtheorem{dfn}[thm]{Definition}
\newtheorem{rem}[thm]{Remark}
\newtheorem{fact}[thm]{Fact}
\newtheorem{obs}[thm]{Observation}

             \def\cF{\mathcal F}

\def\cS{\mathcal S}             \def\cG{\mathcal G}       
             \def\cI{\mathcal I}

\def\N{{\mathbb N}}            \def\Z{{\mathbb Z}}      \def\R{{\mathbb R}}
\def\C{{\mathbb C}}                  
\def\Q{{\mathbb Q}}
\def\1{1\!\!1}
\def\and{{\rm and }}          
         
          \def\Con{\text{Con}}
        \def\diam{\text{\rm {diam}}}

\def\h{\rm{h}}
\def\hmu{\h_\mu}           
\def\H{\text{{\rm H}}}     \def\HD{\text{{\rm HD}}}

\def\Int{\text{{\rm Int}}}
         \def\P{\text{{\rm P}}}     \def\Id{\text{{\rm Id}}}

                   \def\Pa{{\mathcal P}}

\def\a{\alpha}                \def\b{\beta}             \def\d{\delta}
                         \def\f{\phi}
\def\g{\gamma}                           \def\l{\lambda}
              \def\om{\omega}           \def\Om{\Omega}
\def\Sg{\Sigma}               \def\sg{\sigma}
               \def\th{\theta}           
\def\ka{\kappa}

\def\bi{\bigcap}              \def\bu{\bigcup}         \def\du{\bigoplus}
\def\({\bigl(}                \def\){\bigr)}
\def\lt{\left}                \def\rt{\right}

\def\ld{\ldots}               \def\bd{\partial}         \def\^{\tilde}

\def\es{\emptyset}            \def\sms{\setminus}
\def\sbt{\subset}

     \def\imp{\Rightarrow}
\def         \lra{\longrightarrow}  
           
\def\sp{\medskip}             \def\fr{\noindent}        
\def\ov{\overline}            

\def\fr{\noindent}         

\def\om{\omega}

\def\supp{\text{{\rm supp}}}

\begin{document}
\title[]
{ \bf\large Smale endomorphisms over graph-directed Markov systems}
\date{\today}
\author[\sc Eugen MIHAILESCU]{\sc Eugen Mihailescu}

\address{Eugen Mihailescu,
Institute of Mathematics of the Romanian Academy,
P.O Box 1-764,
RO 014700, Bucharest, 
Romania}
\email{Eugen.Mihailescu@imar.ro\newline \hspace*{0.3cm}
Web: www.imar.ro/$\sim$mihailes}
\author[\sc Mariusz URBA\'NSKI]{\sc Mariusz URBA\'NSKI}
\address{Mariusz Urba\'nski, Department of Mathematics,
 University of North Texas, Denton, TX 76203-1430, USA}
\email{urbanski\@unt.edu\newline \hspace*{0.3cm} Web:
www.math.unt.edu/$\sim$urbanski}

\date{}
\thanks{Research of the first author supported in part by grant PN-III-P4-ID-PCE-2016-0823 from UEFISCDI Romania. Research of the second author supported in part by the NSF Grant DMS 0400481.}

\begin{abstract}
We study Smale skew product endomorphisms (introduced in \cite{MU-ETDS}) now over countable graph directed Markov systems, and we prove the exact dimensionality of conditional measures in fibers, and then the global exact dimensionality of the equilibrium measure itself. Our results apply to large classes of systems and have many applications. They apply for instance to natural extensions of graph-directed Markov systems. Another application is to skew products over parabolic systems. We give also applications in ergodic number theory, for example to the continued fraction expansion, and the backward fractions expansion. In the end we obtain a general formula for the Hausdorff (and pointwise) dimension of equilibrium measures with respect to the induced maps of natural extensions $\mathcal T_\beta$ of $\beta$-maps $T_\beta$, for arbitrary  $\beta >1$. 
\end{abstract}
\maketitle
\textbf{MSC 2010:} 37D35, 37A35, 37C45, 37A45, 11K55, 46G10, 60B05.

\textbf{Keywords:} Smale endomorphisms,  equilibrium measures, conditional measures, exact dimensionality, Hausdorff dimension, natural extensions,  stable manifolds,  $\beta$-maps, graph-directed Markov systems, generalized L\"uroth systems, inverse limits,  induced maps.

\

\section{Introduction}
In this paper we extend our study from \cite{MU-ETDS} and now define and investigate  Smale skew product endomorphisms over graph--directed Markov systems with countable alphabets. In this case the limit set of the system may be non-compact and,  due to the fact that the alphabet is countable, the situation is different than in the finite case. \ 

First, in Section \ref{background} we recall several notions and results about Smale skew product endomorphisms  from \cite{MU-ETDS}, necessary in the current paper. 

Then, in Section \ref{S-EMR} we extend the notion of Smale endomorphism to graph-directed Markov systems and prove the exact dimensionality of the conditional measures in fibers for equilibrium states of summable  H\"older continuous potentials. 
Afterwards we show that this implies the global exact dimensionality of the equilibrium measure itself on the skew product. These results are contained in Theorem \ref{Attracting Smale 2} and its Corollaries (compare also with \cite{M-stable}). 

In Subsection \ref{S-parabolicgdms} we apply our results to study skew products with conformal parabolic graph-directed Markov systems in the base. This applies to many classes of examples, such as in Subsection \ref{backwrd} to backward continued fractions (see for eg \cite{AF}).

Next, an important application of our results is given in Section \ref{Natural Extensions of GDMS} to natural extensions of graph-directed Markov systems, and of iterated function systems. These natural extensions are shown to be Smale skew products in fact. We prove the exact dimensionality of conditional measures in their fibers, and then the exact dimensionality of the global equilibrium state. 

Applications of our results are given next in Section \ref{S-GLS}  for generalized L\"uroth systems and their inverse limits. 

Then in Section \ref{beta-t} we apply our results to equilibrium measures for inverse limits of $\beta$-maps for arbitrary $\beta>1$. The natural extensions of $\beta$-maps have complicated structures if $\beta>1$ is arbitrary (see \cite{DKS}). We give a general formula for the pointwise (and Hausdorff) dimension of equilibrium measures with respect to induced maps for natural extensions of $\beta$-maps, first for the case of golden mean $\beta = \frac{1+\sqrt 5}{2}$ in Theorem \ref{sqrt} (and similarly for pseudo-golden mean numbers), and then  for the more difficult case of arbitrary numbers $\beta>1$ in Theorem \ref{beta-global}.

We recall that the notion of Smale skew products and the exact dimensionality of their equilibrium measures was  used in \cite{MU-ETDS}, in order to extend the Doeblin-Lenstra Conjecture, about the approximation coefficients,
$$
\Theta_n(x) := q_n^2\left|x-\frac{p_n}{q_n}\right|
$$ 
of the continued fraction expansion of $x \in [0, 1)$ (see \cite{DK}). The original conjecture, solved in \cite{BJW}, gave the distribution of $\Theta_n(x)$ for Lebesgue-a.e $x \in [0, 1).$ We extended this in \cite{MU-ETDS} for the coefficients $\Theta_n(x)$ associated to numbers $x$ from other sets in $[0, 1)$ which are singular with respect to Lebesgue measure, but have full equilibrium measure. In addition, using our dimension formula, we found the frequency of visits of $\Theta_n(x)$ to small intervals.

\

\section{Background on Smale skew product endomorphisms}\label{background}

In this section we collect some notions and results from \cite{MU-ETDS}, which will be used in the sequel. We include them in order to make the current paper more accessible to the reader. 

\

\subsection{Notions from two-sided thermodynamic formalism}\label{2sidedshift} \ 
Let $E$ be a countable set and a matrix 
$
A:E\times E\lra\{0,1\}
$. \ 
A finite or countable infinite 
tuple $\om$ of elements of $E$ is called $A$-admissible if and only 
if $A_{ab}=1$ for any two consecutive elements $a,b$ of $E$. 
The 
matrix $A$ is called \textit{finitely irreducible} if there is a finite
set $F$ of finite $A$-admissible words so that for any two  
elements $a,b$ of $E$ there exists $\g \in F$ such that
the word $a\g b$ is $A$-admissible. In the sequel, 
the incidence matrix $A$ \textit{is assumed to be finitely irreducible}. \
Given $\b>0$, define the metric $d_\b$ on $E^\N$,
$$
d_\b\((\om_n)_0^{\infty},(\tau_n)_0^{\infty}\)
=\exp\(-\b\max\{n\ge 0:(0\le k\le n) \imp \om_k=\tau_k\}\)
$$
with the standard convention that $e^{-\infty}=0$. All metrics $d_\b$, $\b>0$, on $E^\N$ are H\"older continuously equivalent
and induce the product topology on $E^\N$. Let
$$
E_A^+=\{(\om_n)_0^{\infty}:\forall_{n\in\N}\ A_{\om_n\om_{n+1}}=1\}
$$
 $E_A^+$ is a closed subset of $E^\N$ and we endow it with
the topology and metrics $d_\b$ inherited from $E^\N$. The 
shift map 
$
\sg:E^\N\lra E^\N
$
is defined by 
$
\sg\((\om_n)_0^{\infty}\)=\((\om_{n+1})_{n=0}^{\infty}\),
$
 For every finite word $\om=\om_0\om_1\ld\om_{n-1}$, put $|\om|=n$ the length of $\om$, and 
$$
[\om]=\{\tau\in E_A^+:\forall_{(0\le j\le n-1)}:\tau_j=\om_j\}
$$ 
is the \textit{cylinder} generated by $\om$. If $\psi:E_A^+\to\R$ is continuous, define the pressure $\P(\psi)$ by $$
\P(\psi)
=\lim_{n\to\infty}\frac1n\log\sum_{|\om|=n}\exp\(\sup\(S_n\psi|_{[\om]}\)\)
$$
and the limit exists, as the sequence 
$
\lt(\log\sum_{|\om|=n}\exp\(\sup\(S_n\psi|_{[\om]}\rt)_{n \in \mathbb N}
$
is sub-additive. A function $\psi:E_A^+\lra\R$ is called 
\textit{summable} if and only if 
$$
\sum_{e\in E}\exp\(\sup\(\psi|_{[e]}\)\)< \infty
$$
A shift-invariant
Borel probability measure $\mu$ on $E_A^+$ is called a Gibbs state
of $\psi$ if there are constants $C\ge 1$ and $\P\in\R$
such that
\beq\label{1082305a}
C^{-1}
\le {\mu([\om])\over \exp(S_n\psi(\tau)-\P n)}
\le C
\eeq
for all $n\ge 1$, all admissible words $\om$ of length $n$ and all 
$\tau\in[\om]$. It follows from (\ref{1082305a}) that if $\psi$ admits a Gibbs
state, then $P=\P(\psi)$.

\bdfn\label{1121905}
A function 
$
g:E_A^+\lra\C
$
is called  H\"older continuous if it is H\"older continuous with respect to one, equivalently all, metrics $d_\b$. Then there exists $\b>0$ s.t $g$ is Lipschitz continuous with respect to $d_\b$. The corresponding Lipschitz constant is $L_\b(g)$.
\edfn

\fr A measure realizing supremum in Variational
Principle is called an \textit{equilibrium state} for $\psi$. 
See  also \cite{Sa}, \cite{Sa1} for thermodynamic formalism of equilibrium states on 1-sided countable Markov shifts $E_A^+$. 

We now recall  from \cite{MU-ETDS} results from the thermodynamic formalism of 2-sided shifts $(E_A, \sigma)$ on countable alphabets.
Again $E$ is a countable set and 
$A:E\times E\to\{0,1\}$ a
finitely irreducible matrix. For $\b>0$ the metric $d_\b$
on $E^\Z$ is
$$
d_\b\((\om_n)_{-\infty}^{\infty},(\tau_n)_{-\infty}^{\infty}\)
=\exp\(-\b\max\{n\ge 0:\forall_{k\in\Z}|k|\le n\, \imp\, \om_k=\tau_k\}\)
$$
with  $e^{-\infty}=0$. All
metrics $d_\b$, $\b>0$, on $E^\Z$ induce the product topology on $E^\Z$. We set
$$
E_A=\big\{(\om_n)_{-\infty}^{\infty}:\forall_{n\in\Z}\ A_{\om_n\om_{n+1}}=1\big\}
$$
H\"older continuity is defined similarly as before, for potentials $g: E_A \to \C$.
For every $\om\in E_A$ and all $-\infty\le m\le n\le \infty$, define 
$
\om|_m^n=\om_m\om_{m+1}\ld\om_n.
$
Let $E_A^*$ be the set of all $A$-admissible finite words. For
$\tau\in E^*$, $\tau=\tau_m\tau_{m+1}\ld\tau_n$, we let the cylinder from $m$ to $n$,
$$
[\tau]_m^n=\{\om\in E_A:\om|_m^n=\tau\}
$$
The family of cylinders from $m$ to $n$ is
denoted by $C_m^n$. 
If $m=0$, write $[\tau]$ for $[\tau]_m^n$. 

\sp\fr  Let $\psi:E_A\to\R$ be a continuous function. The topological pressure
$\P(\psi)$ is:
\beq\label{1_2015_11_04}
\P(\psi)
:=\lim_{n\to\infty}\frac1n\log\sum_{\om\in C_0^{n-1}}\exp\(\sup\(S_n\psi|_{[\om]}\)\),
\eeq
and the limit exists by the same subadditivity argument. 
 A shift-invariant
Borel probability measure $\mu$ on $E_A$ is called a \textit{Gibbs state}
of $\psi$ if there are constants $C\ge 1$, $P\in\R$
such that
\beq\label{1082305}
C^{-1}
\le {\mu([\om|_0^{n-1}])\over \exp(S_n\psi(\om)-Pn)}
\le C
\eeq
for all $n\ge 1, \om\in E_A$.  From (\ref{1082305}), if $\psi$ admits a Gibbs
state, then $P=\P(\psi)$. 

The function $\psi:E_A\to\R$ is called \textit{summable} if 
$$
\sum_{e\in E}\exp\(\sup\(\psi|_{[e]}\)\)< \infty
$$

\fr In \cite{MU-ETDS} we  proved the following results:

\blem\label{p2_2015_11_14}\cite{MU-ETDS}
A H\"older continuous $\psi:E_A\to\R$ is summable if and only if $\P(\psi)<\infty$. 
\elem

\bthm\label{t1111105p68}\cite{MU-ETDS}
For every H\"older continuous summable potential $\psi:E_A\to\R$ 
there exists a unique Gibbs state $\mu_\psi$ on $E_A$, and the measure
$\mu_\psi$ is ergodic.
\ethm

\bthm[Variational Principle for Two-Sided Shifts, \cite{MU-ETDS}]\label{t2111105p69}
Suppose that $\psi:E_A\to\R$ is a H\"older continuous summable 
potential. Then
$$
\sup\lt\{\hmu(\sg)+\int_{E_A}\!\!\!\psi d\mu:\mu\circ\sg^{-1}=\mu 
   \  \text{ and }  \  \int\!\!\psi d\mu>-\infty\rt\}
=\P(\psi)
={\hmu}_\psi(\sg)+\int_{E_A}\!\!\psi d\mu_\psi,
$$
and $\mu_\psi$ is the only measure at which this supremum is attained.
\ethm

\fr Equilibrium states for $\psi$ are defined as before for $E_A^+$. 

\

Consider now the partition:
$$
\Pa_-=\{[\om|_0^{\infty}]:\om\in E_A\}
     =\{[\om]:\om\in E_A^+\}
$$
 $\Pa_-$ is a measurable partition of $E_A$ and two
 $\a,\b\in E_A$ are in the same set of $\Pa_-$ if and only if $\a|_0^{\infty}=\b|_0^{\infty}$. If $\mu$
is a Borel probability measure on $E_A$, let 
$$
\{\ov\mu^\tau:\tau\in E_A\}
$$ 
be a \textit{canonical system of conditional measures} induced by
partition $\Pa_-$ and measure $\mu$ (see Rokhlin \cite{Ro}). Each $\ov\mu^\tau$ is a Borel probability
measure on $[\tau|_0^{\infty}]$ and we also write $\ov\mu^\om$, $\om\in E_A^+$, for the 
conditional measure on $[\om]$. The canonical projection is:
$$
\pi_0: E_A \to E_A^+, \  \pi_0(\tau) = \tau|_0^\infty, \tau \in E_A,
$$
The system $\{\ov\mu^\om:\om\in E_A^+\}$ of conditional measures is determined by the property that,
$$
\int_{E_A}g\,d\mu=\int_{E_A^+}\int_{[\om]}g\,d\ov\mu^{\om}
  \,d(\mu\circ \pi_0^{-1})(\om)
$$
for every $g\in L^1(\mu)$ (\cite{Ro}). The canonical system of conditional measures of $\mu$ is uniquely defined up to a set of $\mu\circ \pi_0^{-1}$-measure zero.

\bthm\label{t4111105p70}\cite{MU-ETDS}
Suppose that $\psi:E_A\to\R$ is a H\"older continuous summable 
potential. Let $\mu$ be a Borel probability shift-invariant measure on
$E_A$. Then $\mu=\mu_\psi$, the unique Gibbs state for $\psi$ if
and only if there exists $D\ge 1$ such that for all $n \ge 1$,
\beq\label{1111105}
D^{-1}
\le {\ov\mu^\om([\tau\om])\over 
    \exp\(S_n\psi(\rho)-\P(\psi)n\)}
\le D,
\eeq
for $\mu\circ \pi_0^{-1}$-a.e $\om\in E_A^+$, $\ov\mu^\om$-a.e $\tau\om\in E_A(-n,\infty)$ with
$A_{\tau_{-1}\om_0}=1$, and  $\rho\in\sg^{-n}([\tau\om|_{-n}^{\infty}])=
[\tau\om|_0^{\infty}]$.

\ethm



\subsection{Skew product Smale spaces of countable type}\label{con}

\begin{dfn}\label{Smaleskew}\cite{MU-ETDS}
Let $(Y,d)$ be a complete bounded metric space. For every $\om
\in E_A^+$ let $Y_\om\sbt Y$ be an arbitrary set and let
$
T_\om:Y_\om\lra Y_{\sg(\om)}
$
be a continuous injective map. Define
$$
\hat Y:=\bu_{\om\in E_A^+}\{\om\}\times Y_\om\sbt E_A^+\times Y
$$
Define the map 
$
T:\hat Y\lra\hat Y
$
by 
$
T(\om,y)=(\sg(\om),T_\om(y)).
$
The pair $(\hat Y,T:\hat Y\to\hat Y)$ is called a \textit{skew product 
Smale endomorphism} if there exists $\l>1$ such 
that  $T$ is fiberwise uniformly contracting, i.e for all $\om\in E_A^+$ and all $y_1,y_2\in Y_\om$, 
\beq\label{1111205}
d(T_\om(y_2),T_\om(y_1))\le \l^{-1}d(y_2,y_1)
\eeq
\end{dfn}

For every $\tau\in E_A(-n,\infty)$ the map 
$
T_\tau^n
=T_{\tau|_{-1}^{\infty}}\circ T_{\tau|_{-2}^{\infty}}\circ\ld
   \circ T_{\tau|_{-n}^{\infty}}:Y_\tau\lra Y_{\tau|_0^{\infty}}
$
is well-defined. For every $\tau\in E_A$ define 
$$
T_\tau^n
:=T_{\tau|_{-n}^{\infty}}^n
:=T_{\tau|_{-1}^{\infty}}\circ T_{\tau|_{-2}^{\infty}}\circ\ld
  \circ T_{\tau|_{-n}^{\infty}}:Y_{\tau|_{-n}^{\infty}}\longrightarrow Y_{\tau|_0^{\infty}}
$$
Then  $\(T_\tau^n\(Y_{\tau|_{-n}^{\infty}}\)\)_{n=0}
^\infty$ are descending, and 
$
\diam\(\ov{T_\tau^n\(Y_{\tau|_{-n}^{\infty}}\)}\)\le \l^{-n}\diam(Y).
$
As $(Y,d)$ is complete, $
\bi_{n=1}^\infty \ov{T_\tau^n\(Y_{\tau|_{-n}^{\infty}}\)}
$
is a singleton denoted by $\hat\pi_2(\tau)$. Hence we  defined the map
$$
\hat\pi_2:E_A\lra Y,
$$
and  define also $\hat\pi: E_A\to E_A^+\times Y$ by 
\beq\label{5111705p141}
\hat\pi(\tau)=\(\tau|_0^{\infty},\hat\pi_2(\tau)\),
\eeq
and the truncation to the elements of non-negative indices by 
$$\pi_0: E_A \lra E_A^+, \  \  \pi_0(\tau) = \tau|_0^\infty$$
In the notation for $\pi_0$ we drop the hat symbol, as this projection is independent of the skew product on $\hat Y$.
For all $\om\in E_A^+$ define the $\hat \pi_2$-projection of the cylinder $[\om]\subset E_A$, 
$$
J_\om:=\hat\pi_2([\om])\in Y,
$$
and call these sets the \textit{stable Smale fibers} of the system $T$. The global invariant set is:
$$
J:=\hat\pi(E_A)=\bu_{\om\in E_A^+}\{\om\}\times J_\om\sbt E_A^+\times Y,
$$
called the \textit{Smale space} (or the \textit{fibered limit set}) induced by the Smale pre-system $T$.

For each $\tau\in E_A$ we have  $\hat\pi_2(\tau)\in \ov Y_{\tau|_0^{\infty}}$; so
$
J_\om\sbt \ov Y_\om
$
for every $\om\in E_A^+$. Since all maps $T_\om:Y_\om\to Y_{\sg(\om)}$ are Lipschitz continuous with Lipschitz constant $\l^{-1}$, they extend uniquely to Lipschitz continuous maps  from $\ov Y_\om$ to $\ov Y_{\sg(\om)}$. In \cite{MU-ETDS} we have proved the following.

\bprop\label{pt1.1}\cite{MU-ETDS}
For every $\om\in E_A^+$ we have that
$
T_\om(J_\om)\sbt J_{\sg(\om)},$ 
$
\bu_{e\in E,  A_{e\om_0}=1}
T_{e\om}(J_{e\om})=J_\om, 
$
and
$
T\circ\hat\pi=\hat\pi\circ\sg
$.
\eprop
 Then $T(J)\sbt J$, so consider the system
$$
T:J\lra J, \ T(\omega, y) = (\sigma(\omega), T_\omega(y)),
$$
which we call the \textit{skew product Smale endomorphism} generated by  $T:\hat Y\lra \hat Y$.

\

If $Y$ is the closure of a bounded open set in $\R^D$, then at any point $(\omega, y) \in J$, there is a local stable manifold  $\{\omega\} \times D(y)$, where $D(y)$ is a small ball $B(y, r) \subset Y$.

\bdfn\label{d3t13}
A measure $\mu\in M_T(J)$ is called an equilibrium state of the continuous potential $\psi:\hat Y\to\R$, if $\int\psi\,d\mu>-\infty$ and \
$
\hmu(T)+\int_J\psi\,d\mu= \P_T(\psi).
$
\edfn

\bdfn\label{d3t13a}
The potential $\psi:J\to\R$ is called \textit{summable} if 
$$
\sum_{e\in E}\exp\(\sup(\psi|_{[e]_T})\)<\infty
$$
\edfn

\bobs\label{o1t13.1}
$\psi:J\to\R$ is summable if and only if $\psi\circ\hat\pi:E_A\to\R$ is summable.
\eobs

\bdfn\label{d1t13}
We call a  continuous skew product Smale endomorphism $T:\hat Y\to\hat Y$ H\"older, if  $\hat\pi:E_A\to J$ is H\"older continuous. 
\edfn

\subsection{Conformal skew product Smale endomorphisms}\label{CSPSE}
In this subsection we keep the setting of skew product Smale endomorphisms. As in \cite{MU-ETDS}, we assume more about the spaces $Y_\om$, $\om\in E_A^+$, and the fiber maps $T_\om:Y_\om\to Y_{\sg(\om)}$, namely:
\begin{itemize}
\item[(a)] $Y_\om$ is a closed bounded subset of $\R^d$, with some $d\ge 1$ such that $\ov{\Int(Y_\om)}=Y_\om$. 

\item[(b)] Each map $T_\om:Y_\om\to Y_{\sg(\om)}$ extends to a $C^1$ conformal embedding from $Y_\om^*$ to $Y_{\sg(\om)}^*$, where $Y_\om^*$ is a bounded connected open subset of $\R^d$ containing $Y_\om$. The same symbol $T_\om$ denotes this extension and we assume that the maps
$
T_\om:Y_\om^*\to Y_{\sg(\om)}^*$ satisfy:
\item[(c)] Formula \eqref{1111205}  holds for all $y_1, y_2\in Y_\om^*$, perhaps with some smaller constant $\l>1$.
\item[(d)] (Bounded Distortion Property 1) There exist constants $\a>0$ and $H>0$ such that for all $y, z\in Y_\om^*$ we have that: \ 
$$
\big|\log|T_\om'(y)|-\log|T_\om'(z)|\big|\le H||y-z||^\a
$$
\item[(e)] The function 
$
E_A\ni\tau\longmapsto\log|T_\tau'(\hat\pi_2(\om))|\in\R
$
is H\"older continuous.
\item[(f)] (Open Set Condition) For every $\om\in E_A^+$ and for all $a, b\in E$ with $A_{a\om_0}=A_{b\om_0}=1$ and $a\ne b$, we have \ 
$
T_{a\om}(\Int(Y_{a\om}))\cap T_{b\om}(\Int(Y_{b\om}))=\es.
$
\item[(g)] (Strong Open Set Condition) There exists a measurable function $\d:E_A^+\to(0,\infty)$ such that for every $\om \in E_A^+$, \ 
$
J_\om\cap\(Y_\om\sms\ov B(Y_\om^c,\d(\om)\)\ne\es.
$
\end{itemize}

\fr Any skew product Smale endomorphism satisfying conditions (a)--(g) will be called in the sequel a \textit{conformal skew product Smale endomorphism}. 

From (c), (d), and (e), we obtain two more Bounded Distortion Properties (BDP):

\fr (BDP 2)  There is $H>0$, so that for all $\tau\in E_A,  \ y,z\in Y_{\tau|_{-n}^{\infty}}^*, n >0$,
$$
\Big|\log\big|\(T_\tau^n\)'(y)\big|-
\log\big|\(T_\tau^n\)'(z)\big|\Big|
\le H||y-z||^\a
$$
\fr (BDP 3) For all $\tau\in E_A$, $n\ge 0$, and $y,z\in Y_{\tau|_{-n}^{\infty}}^*$, 
$$
K^{-1}
\le \frac{\big|\(T_\tau^n\)'(y)\big|}
{\big|\(T_\tau^n\)'(z)\big|}
\le K.
$$

\subsection{General skew products over countable--to--1 endomorphisms.}\label{generalskew} 

We now recall some notions and results on general skew products over countable-to-1 endomorphisms from \cite{MU-ETDS}, that will be used in the sequel.
 For dynamics and thermodynamic formalism for various types of endomorphisms, one can see  \cite{Ru}, \cite{ABN}, \cite{BT}, \cite{M-stable}, \cite{M-DCDS12}, \cite{M-ETDS11}, \cite{M-MZ}, \cite{MS}, \cite{MU-JFA}, etc.  In \cite{M-stable} it was proved a result about the exact dimensionality of measures on stable manifolds for hyperbolic endomorphisms. \ 
In \cite{MU-ETDS} we proved a result about skew products whose base transformations are modeled by 1-sided shifts on a countable alphabet.
Assume we have a skew product:
$$
F: X \times Y \lra X \times Y,
$$
where $X$ and $Y$ are complete bounded metric spaces, $Y \subset \R^d$ for some $d \ge 1$, and 
$$
F(x, y) = (f(x), g(x, y)),
$$ 
where the map 
$
Y\ni y \longmapsto g(x, y)
$
is injective and continuous for every $x \in X$.\ 
Denote the map 
$
Y\ni y \longmapsto g(x, y)
$
also by $g_x(y)$. Assume that 
$$
f:X \lra X
$$ 
is at most countable--to--1, with dynamics modeled by a 1--sided Markov shift on a countable alphabet $E$ with matrix $A$ finitely irreducible, i.e there is a surjective H\"older continuous
$$
p: E_A^+ \lra  X 
$$
called \textit{coding}, so that
$
p\circ \sigma = f\circ p.
$  
Assume conditions (a)--(g) from \ref{CSPSE} are satisfied for,
$$
T_\om:Y_\om\lra Y_{\sg(\om)}, \
T_\omega := g_{p(\omega)},
$$
for all $\om \in E_A^+$. Then as in \cite{MU-ETDS} we call
$
F: X \times Y \lra X \times Y
$
a \textit{generalized conformal skew product Smale endomorphism. }

Given the skew product $F$ as before, we form a  skew product endomorphism by
defining for every $\om \in E_A^+$, the fiber map $\hat F_\om: Y \to Y$,
$$
\hat F_\om(y) := g(p(\om),y)
$$
The system $(\hat Y, \hat F)$ is called \textit{the symbolic lift} of $F$.
If $\hat Y = E_A^+ \times Y$, we obtain a conformal skew product Smale endomorphism $\hat F: \hat Y \to \hat Y$ given by
\begin{equation}\label{skew-general}
\hat F(\omega, y) = (\sigma(\om), \hat F_\om(y))
\end{equation}

As in subsection~\ref{con}, we study the fibers $J_\om, \ \om \in E_A^+$ and the sets $J_x$, $x\in X$. From definition, $J_\om  = \hat\pi_2([\om])$ is the set of points of type 
$$
\bi_{n \ge 1}\ov{\hat F_{\tau_{-1}\om}\circ \hat F_{\tau_{-2}\tau_{-1}}\om\circ\ldots \circ \hat F_{\tau_{-n}\ldots\tau_{-1}\om}(Y)}
$$
By $n$-\textit{prehistory} of $x$ with respect to $(f, X)$ we understand any finite set of points in $X$
$$
(x, x_{-1}, x_{-2}, \ldots, x_{-n})\in X^{n+1},
$$
where 
$
f(x_{-1}) = x,\, f(x_{-2}) = x_{-1}, \ldots, f(x_{-n}) = x_{-n+1}.
$
Call a \textit{prehistory} of $x$ with respect to  $(f, X)$, any infinite sequence of consecutive preimages in $X$, i.e. 
$
\hat x = (x, x_{-1}, x_{-2}, \ldots),
$
where  
$
f(x_{-i}) = x_{-i+1},
$
$i \ge -1$.
\newline
The space of prehistories is denoted by $\hat X$ and is called the \textit{natural extension} (or \textit{inverse limit}) of  $(f, X)$. It projects onto $X$ by $\pi(\hat x) = x, \  \hat x \in \hat X$. There is a bijection $\hat f: \hat X \to \hat X$,
$$
\hat f(\hat x) = (f(x), x, x_{-1}, \ldots)
$$
The terms inverse limit and natural extension are both used in the sequel, without having necessarily an invariant measure on $X$. 
On $\hat X$ we take the canonical metric, which induces the topology equivalent to the one inherited from the product topology on $X^{\N}$. Then $\hat f$ is a homeomorphism. For more on dynamics of endomorphisms and inverse limits, one can see for eg, \cite{DKS}, \cite{Ru}, \cite{M-DCDS12}, \cite{FM}, \cite{M-stable}, \cite{MS}, \cite{M-MZ}, \cite{M-JSP11}, \cite{M-ETDS11}, \cite{M-Mon}. 
In above notation,
$
f(p(\tau_{-1}\om)) = p(\om)=x,
$
and for all the prehistories of $x$, $\hat x = (x, x_{-1}, x_{-2}, \ldots) \in \hat X$, consider the set $J_x$ of points of type 
$$\bigcap_{n \ge 1} \ov{g_{x_{-1}} \circ g_{x_{-2}}\circ \ldots \circ g_{x_{-n}}(Y)}$$

\fr If $\hat \eta = (\eta_0, \eta_1, \ldots)$ is another sequence in $E_A^+$ with $p(\hat\eta) = x$, then for any $\eta_{-1}$ with $\eta_{-1}\hat\eta \in E_A^+$, we have $p(\eta_{-1}\hat \eta) = x_{-1}'$ where $x_{-1}'$ is a 1-preimage (i.e preimage of order 1) of $x$. 
Thus,  
\begin{equation}\label{J_x}
J_x = \mathop{\bigcup}\limits_{\omega\in E_A^+, p(\omega) = x} J_\omega
\end{equation}

We denote the fibered limit sets for $T$ and $F$ by: 
\begin{equation}\label{J(X)}
J = \bigcup_{\omega \in E_A^+} \{\omega\}\times J_\omega \subset E_A^+ \times Y  \  \ \text{and} \  \  J(X):= \bigcup_{x \in X}\{x\}\times J_x \subset X \times Y
\end{equation}
So
$
\hat F(J)=J$
  and $ 
F(J(X))=J(X).
$
The H\"older continuous projection is:
$$
p_J: J \lra J(X), \ 
p_J(\omega, y) = (p(\omega), y),
$$
and we obtain 
$
F\circ p_J = p_J \circ \hat F. 
$
In the sequel, 
$
\hat\pi_2:E_A\lra Y
\  \  \  {\rm and} \  \  \  
\hat\pi:E_A\lra E_A^+\times Y
$
are the maps defined in subsection~\ref{con} and,
$$
\hat \pi(\tau) = (\tau|_0^\infty, \hat\pi_2(\tau))
$$

Now, we want to know if enough points $x \in X$ have unique coding sequences in $E_A^+$. 

\begin{dfn}\label{phi-inj}\cite{MU-ETDS}
Let $F: X \times Y \lra X \times Y$ be a generalized conformal skew product Smale endomorphism. Let $\mu$ be a Borel probability measure $X$. We then say that the coding 
$$
p:E_A^+ \lra X
$$ 
is $\mu$--injective, if there exists a $\mu$-measurable set $G \subset X$ with $\mu(G) = 1$ such that for every point $x \in G$, the set $p^{-1}(x)$ is a singleton in $E_A^+$. 

Denote such a set $G$ by $G_\mu$ and for $x\in G_\mu$ the only element of $p^{-1}(x)$ by $\om(x)$.
\end{dfn}

In \cite{MU-ETDS} we proved the following.

\begin{prop}\label{phi-i}
If the coding $p:E_A^+ \lra X$ is $\mu$--injective, then for every $x\in G_\mu$, we have that
$
J_x = J_{\omega(x)}.
$
\end{prop}

In the sequel we work \textit{only with $\mu$-injective codings}, and the measure $\mu$ will be clear from the context. Also given a metric space $X$ with a coding $p:E_A^+ \to X$, and a potential 
$
\phi:X \lra \mathbb R,
$
we say that $\phi$ is \textit{ H\"older continuous} if and only if $\phi\circ p$ is  H\"older continuous.

\

Now consider a potential $\phi: J(X) \lra \R$ such that the potential 
$$
\widehat\phi: = \phi\circ p_J\circ \hat \pi: E_A \to \R
$$
is H\"older continuous and summable. For example, $\widehat\phi$ is  H\"older continuous if $\phi: J(X) \to \R$ is itself  H\"older continuous. This case will be quite frequent in examples below. Define now
\beq\label{1_2017_01_13}
\mu_\phi := \mu_{\widehat\phi} \circ (p_J\circ \hat \pi)^{-1},
\eeq
and call it the equilibrium measure of $\phi$ on $J(X)$ with respect to $F$. 

\sp Now, let us consider the partition $\xi'$ of $J(X)$ into the fiber sets $\{x\} \times J_x, \ x\in X$, and the conditional measures 
$\mu^x_\phi$ associated to $\mu_\phi$ with respect to the measurable partition $\xi'$ (see \cite{Ro}). Recall that for each $\omega \in E_A^+$, we have \ 
$
\hat\pi_2([\omega]) = J_\omega.
$

\bdfn\label{EDim}
We say in general that  a measure $\mu$ is exact dimensional on a space $X$, if its \textit{pointwise dimension} at $x$, defined by the formula $$d_\mu(x) = \mathop{\lim}\limits_{r\to 0} \frac{\log \mu(B(x, r))}{\log r},$$ exists for $\mu$-a.e $x\in X$, and $d_\mu(\cdot)$ is constant $\mu$-almost everywhere.
\edfn

 In see \cite{Pe}, \cite{BPS}, \cite{Y} were proved several results about exact dimensionality of invariant measures on manifolds. Also in \cite{M-stable} the conditional measures on stable manifolds of hyperbolic endomorphisms were proved to be geometric and exact dimensional.
 
 If $\mu$ is a probability $\sg$-invariant measure on $E_A$, then by $\chi_\mu(\sg)$ is its \textit{Lyapunov exponent},
$$
\chi_\mu(\sg)
:=-\int_{E_A}\log\Big|T_{\tau|_0^{\infty}}'\(\hat\pi_2(\tau)\)\Big|\,d\mu(\tau)
 =-\int_{E_A^+}\int_{[\om]}\log\big|T_\om'\(\hat\pi_2(\tau)\)\big|\,d\ov\mu^\om
   (\tau)\,dm(\om),
$$
where $m=\mu\circ \pi_0^{-1}= \pi_{1*}\mu$ is the canonical projection of $\mu$ onto $E_A^+$.

Define the \textit{Lyapunov exponent} of an $F$-invariant measure $\mu$ on  $J(X) = \bigcup\limits_{x\in X} \{x\}\times J_x$ by: 
$$
\chi_\mu(F) = - \int_{J(X)} \log|g_x'(y)| \,  d\mu(x,y)
$$

Denote by $p_1:X\times Y\to X$ the canonical projection on the first coordinate,
$
p_1(x,y)=x.
$ Then we proved that the conditional measures on fibers are exact dimensional.

\begin{thm}\label{codedskews}\cite{MU-ETDS}
Let $F: X \times Y \lra X \times Y$ a generalized conformal skew product Smale endomorphism. Let $\phi:J(X) \lra \R$ be a potential such that 
$
\hat\phi:= \phi \circ p_J\circ \hat \pi: E_A \lra \R
$
is  H\"older continuous summable. 
Assume the coding $p:E_A^+ \to X$ is $\mu_\phi\circ p_1^{-1}$--injective.

\fr Then, for $\mu_\phi\circ p_1^{-1}$-a.e $x \in X$, the conditional measure $\mu_\phi^x$ is exact dimensional on $J_x$, and
$$
\lim_{r\to 0} \frac{\log \mu_\phi^x(B(y, r))}{\log r} 
= \frac{\h_{\mu_\phi}(F)}{\chi_{\mu_\phi}(F)}
=\HD(\mu_\phi^x),
$$
for $\mu_\phi^x$-a.e $y \in J_x$; hence equivalently for $\mu_\phi$-a.e $(x, y) \in J(X)$.
\end{thm}


By using Theorem \ref{codedskews}, we have proved in \cite{MU-ETDS} the exact dimensionality of conditional measures of equilibrium states on fibers for many types of skew products. 
Then we also proved the following  result about \textit{global} exact dimensionality of measures on  $J(X)$. 

\begin{thm}\label{globalexact}\cite{MU-ETDS} \
Let $F: X \times Y \lra X \times Y$ a generalized conformal skew product Smale endomorphism. Assume that $X\sbt\R^d$ with some integer $d\ge 1$. Let $\mu$ be a Borel probability $F$--invariant measure on $J(X)$, and $(\mu^x)_{x\in X}$ be the Rokhlin's canonical sytem of conditional measures of $\mu$, with respect to the partition $\(\{x\}\times J_x\)_{x\in X}$. Assume that:
\begin{itemize}
\item[(a)]  There exists $\a>0$ such that for $\mu\circ p_1^{-1}$-a.e $x \in X$, the conditional measure $\mu^x$ is exact dimensional and $\HD(\mu^x)=\a$, 

\item[(b)] The measure $\mu\circ p_1^{-1}$ is exact dimensional on $X$. 
\end{itemize}

 \fr Then, the measure $\mu$ is exact dimensional on $J(X)$, and for $\mu$-a.e $(x,y) \in J(X)$,
$$
\HD(\mu)=\lim\limits_{r\to 0} \frac{\log \mu(B((x, y), r))}{\log r} = \alpha + \HD(\mu\circ p_1^{-1})
$$
\end{thm}

\

\section{Skew Products with the Base Maps Being Graph--Directed Markov Systems}\label{S-EMR}

In this section we consider dynamical systems being skew products for which the base map is given by a countable alphabet conformal graph directed Markov system (GDMS). Our main goal is to prove that equilibrium measures for skew products over such base maps are exact dimensional. A \emph{directed multigraph} consists of:
\begin{itemize} 
\item A finite set $V$ of vertices, 
\item A countable (either finite or infinite) set $E$ of directed edges, 
\item A map $A:E\times E\to \{0,1\}$ called an \emph{incidence matrix} on $(V,E)$,\
\item Two functions $i,t:E\to V$, such that $A_{ab} = 1$ implies $t(a) = i(b)$. 
\end{itemize}

\fr Now suppose also that a collection of nonempty compact metric spaces $\{X_v\}_{v\in V}$ is given along with a number $\lambda\in (0,1)$, and that for every $e\in E$, we have an injective contraction 
$$
\phi_e:X_{t(e)}\lra X_{i(e)}
$$ 
with Lipschitz constant $\le \lambda$. Then the collection
\[
\cS = \big\{\phi_e:X_{t(e)}\lra X_{i(e)}\big\}_{e\in E}
\]
is called a \emph{graph directed Markov system} (or \emph{GDMS}). We now describe the limit set of  $\cS$. For every $\omega \in E^+_A$,  $\{\phi_{\omega|_n}\left(X_{t(\omega_n)}\right)\}_{n \geq 1}$ form a descending sequence of compact sets and thus
$
\bigcap_{n \geq 1}\phi_{\omega|_n}\left(X_{t(\omega_n)}\right)\ne\emptyset.
$
Since for every $n \geq 1$, 
$
\diam\left(\phi_{\omega|_n}\left(X_{t(\omega_n)}\right)\right)\le \lambda^n\diam\left(X_{t(\omega_n)}\right)\le \lambda^n\max\{\diam(X_v):v\in V\},
$
 the intersection 
$$
\bigcap_{n \in \N}\phi_{\omega|_n}\left(X_{t(\omega_n)}\right)
$$
is a singleton, and we denote its only element by $\pi(\omega)$. In this way we  define a map
\[
\pi:E^+_A\lra \coprod_{v\in V}X_v,
\]
where 
$
X:=\coprod_{v\in V} X_v
$
is the disjoint union. The map $\pi$ is called the \emph{coding map}, and
\[
J = J_\cS := \pi(E^+_A)
\]
is called the \emph{limit set} of the GDMS $\cS$. The sets
$$
J_v = \pi\(\{\omega \in E_A^+: i(\omega_1) = v\}\)
$$
for $v\in V$ are called the \emph{local limit sets} of $\cS$.

\sp We call the GDMS $\cS$ \emph{finite} if the alphabet $E$ is finite. Moreover we call $\cS$ \emph{maximal} if for all $a,b\in E$, we have $A_{ab}=1$ if and only if $t(a)=i(b)$.  In \cite{gdms} a maximal GDMS was called a \emph{graph directed system} (abbr. GDS).
Finally, we call a maximal GDMS $\cS$ an \emph{iterated function system} (or \emph{IFS}) if $V$, the set of vertices of $\cS$, is a singleton. Equivalently, a GDMS is an IFS if and only if the set of vertices of $\cS$ is a singleton and all entries of the incidence matrix $A$ are $1$. For finite IFS with overlaps and exact dimensionality of measures see \cite{FH}. For thermodynamic formalism for (countable) IFS with overlaps, and exact dimensionality of measures, see \cite{MU-Adv}, \cite{MU-ETDS}, \cite{MU-JSP16}, \cite{MU-CM2013}. Our setting below is different however, in the sense that it uses a different type of randomization.

\begin{dfn}\label{definitionsymbolirred}
We call the GDMS $\cS$ and its incidence matrix $A$ \emph{finitely (symbolically) irreducible} if there exists a finite set $\Lambda\subset E_A^*$ such that for all $a, b\in E$ there exists a word $\omega\in\Lambda$ such that the concatenation $a\omega b$ is in $E_A^*$. $\cS$ and $A$ are called \emph{finitely primitive} if the set $\Lambda$ may be chosen to consist of words all having the same length. Note that all IFSs are finitely primitive.
\end{dfn}

We call a GDMS \emph{conformal} if for some $d\in\N$, the following conditions hold:

\begin{itemize}
\item[(a)] For every vertex $v\in V$, $X_v$ is a compact connected subset of $\R^d$, and $X_v=\overline{\Int(X_v)}$.

 \item[(b)] There exists a family of open connected sets $W_v \subset X_v \;\; (v\in V)$ such that for every $e\in E$, the map $\phi_e$ extends to a $C^1$ conformal diffeomorphism from $W_{t(e)}$ into $W_{i(e)}$ with Lipschitz constant $\leq \lambda$.

\item[(c)] (Bounded Distortion Property) There are two constants $L\ge 1$ and $\alpha>0$ such that for every $e\in E$ and every pair of points $x,y\in X_{t(e)}$,
\[
\left|\frac{|\phi_e'(y)|}{|\phi_e'(x)|}-1 \right| 
\le L\|y-x\|^\alpha,
\]
where $|\phi_\omega'(x)|$ denotes the scaling of the derivative, which is a linear similarity map.

\item[(d)] (Open Set Condition) If $a, b\in E$, and $a\ne b$, then \
$
\phi_a(\Int(X_a))\cap \phi_b(\Int(X_b))=\es.
$
\item[(e)] (Boundary Condition)
There exists $e\in E$ such that
$$
J_\cS\cap \Int X_e\ne\es
$$
If the Open Set Condition and the Boundary Condition are both satisfied, then we say that the Strong Open Set Condition (SOSC) is satisfied. 
\end{itemize}

\

Examples of conformal GDMS contain: GDMS obtained by \textit{gluing} conformal iterated function systems in the same Euclidean space; \ and GDMS obtained from Markov partitions $\{X_m\}_{m\in M}$ of \textit{expanding} conformal maps $f:X \to X$ with $M$ as the set of vertices, and contractions being the inverse branches of $f$, and incidence matrix $A$ given by the respective Markov partition $\{X_m\}_{m\in M}$.

We define now the GDMS map 
$$
f=f_\cS:J_\cS\lra J_\cS,
$$
associated to the system $\cS$, by 
\beq\label{1_2016_12_16}
f(\phi_e(x)):=x
\eeq
if $x\in\Int(X_{t(e)})$ (then $e$ is uniquely determined), and $f(z)$ to be some given preassigned point $\xi$ of $J_\cS$, if $z\notin \bu_{e\in E} \phi_e\(\Int(X_{t(e)})\)$. 

\sp A special class of conformal GDMSs is provided by one dimensional--systems. Indeed, if $X$ is compact interval in $\R$, then the GDMSs, more precisely the derived maps $f_\cS:J_\cS\to J_\cS$, associated to them, are sometimes called \textit{expanding Markov-R\'enyi maps (EMR maps)}; see  \cite{PoW}. A sufficient condition for (BDP), i.e. (c) is the \textit{R\'enyi condition}, i.e,
$$
\sup\limits_{e\in E}\, \sup\limits_{x, y, z \in X}\lt\{\frac{|\phi_e''(x)|}{|\phi_e'(y)|\cdot |\phi_e'(z)|}\rt\}< \infty
$$

\sp Let us now consider a general GDMS map $f:J_\cS \to J_\cS$, and a skew product 
$$
F:J_\cS \times Y \lra J_\cS \times Y,
$$
where $Y\subset \R^d$ is a bounded open set, defined by the formula 
$$
F(x, y) = \(f(x), g(x,y)\), 
$$
Recall from ~\ref{generalskew} that the symbolic lift of $F$ is $\hat F: E_A^+\times Y \lra E_A^+\times Y$, given by 
$$
\hat F(\omega, y) = (\sigma(\omega), g(\pi(\omega),y))
$$
The map $p:E_A^+\to X$ is now equal to the map $\pi_\cS:E_A^+\to X$. So, the map
$
p\times\Id:E_A^+\times Y \to J_\cS \times Y
$
is given by the formula
$
(p\times\Id)(\om,y):=(\pi_\cS(\om),y).
$
Using the notation of subsection~\ref{generalskew}, we denote by $p_J$ its restriction to the set 
$$
J=\bu_{\om\in E_A^+}\{\om\}\times J_\om
$$ 
If the symbolic lift $\hat F$ is a H\"older conformal skew product Smale endomorphism, then we say by extension that $F$ is a \textit{H\"older conformal skew product endomorphism over $f$}.
Recall also from (\ref{J(X)}) that the fibered limit set of $F$ is 
$$
J(J_\cS) = \bu\limits_{x \in J_\cS} \{x\} \times J_x
$$ 
The first result, easy but useful for us is the following.

\blem\label{l1_2017_02_20}
Let $f:J_\cS\lra J_\cS$ be a finitely irreducible conformal GDMS map, let $Y \subset \R^d$ be an open bounded set, and let 
$$
F:J_\cS \times Y \lra J_\cS \times Y
$$ 
be a H\"older conformal skew product endomorphism over $f$. If $\nu$ is a Borel probability shift--invariant ergodic measure on $E_A^+$ with full topological support, then the coding 
$
p=\pi_\cS:E_A^+\lra J_\cS
$
is $\nu\circ\pi_\cS^{-1}$--injective. 
\elem

\bpf
Since $\phi_e(\Int X_{t(e)})\sbt \Int X_{t(e)}$, we have that 
$
\sg^{-1}\(\pi_\cS^{-1}(\Int X)\)\sbt \pi_\cS^{-1}(\Int X),
$
where
$$
\Int X:=\bu_{e\in E}\Int X_{t(e)}
$$
Since the Borel probability measure $\nu$ is shift--invariant and ergodic, it thus follows that
$
\nu\(\pi_\cS^{-1}(\Int X)\)\in\{0,1\}
$.
But since $\supp(\nu)=E_A^+$, we thus conclude from the Strong Open Set Condition (e) that $\nu\(\pi_\cS^{-1}(\Int X)\)>0$. Hence,
$
\nu\(\pi_\cS^{-1}(\Int X)\)=1
$. 
From the shift--invariance of $\nu$, we thus conclude that
$$
\nu\lt(\bi_{n=0}^\infty\sg^{-n}\(\pi_\cS^{-1}(\Int X)\)\rt)=1
$$
Denote the set in parentheses by $\Int_\infty(\cS)$. Then, by the Open Set Condition (d), the map $\pi_\cS|_{\Int_\infty(\cS)}$ is one--to--one and 
$
\pi_\cS^{-1}\(\pi_\cS(\Int_\infty(\cS))\)=\Int_\infty(\cS)
$.
Thus,
$$
\nu\circ\pi_\cS^{-1}\(\pi_\cS(\Int_\infty(\cS))\)
=\nu\(\Int_\infty(\cS)\)
=1,
$$
and for every point $x\in\pi_\cS(\Int_\infty(\cS))$, the set $\pi_\cS^{-1}(x)$ is a singleton. 

\end{proof}

We can now prove the following result:

\begin{thm}\label{EMRskews}
Let $f:J_\cS\lra J_\cS$ be a finitely irreducible conformal GDMS map, let $Y \subset \R^d$ be open bounded, and let a H\"older conformal skew--product endomorphism over $f$,
$$
F:J_\cS \times Y \lra J_\cS \times Y
$$ 
Let $\phi:J(J_\cS) \to \R$ be a potential such that 
$
\widehat\phi= \phi \circ p_J\circ \hat \pi: E_A \lra \R
$
is a locally H\"older continuous summable potential on $E_A$. 
Then, for $\mu_\phi\circ p_1^{-1}$--a.e $x \in X$, the conditional measure $\mu_\phi^x$ is exact dimensional on $J_x$, and, for $\mu_\phi^x$-a.e. $y \in J_x$,
$$
\lim_{r\to 0} \frac{\log \mu_\phi^x(B(y, r))}{\log r} 
= \frac{\h_{\mu_\phi}(F)}{\chi_{\mu_\phi}(F)}
=\HD(\mu_\phi^x)
$$

\end{thm}

\bpf
One needs to notice that $\supp(\mu_{\widehat\phi}\circ\pi^{-1})=E_A^+$ since $\mu_{\widehat\phi}$ is the equilibrium state of the locally H\"older continuous summable potential $\widehat\phi$ on $E_A$. We showed in \cite{MU-ETDS} that 
$$\mu_\phi\circ p_1^{-1}
= \mu_{\widehat\phi}\circ (p_J\circ \hat \pi)^{-1}\circ p_1^{-1}
= \mu_{\widehat\phi}\circ (p_1\circ p_J\circ \hat \pi)^{-1}            
= \mu_{\widehat\phi}\circ (p\circ\pi_0)^{-1}
= \mu_{\widehat\phi}\circ \pi_0^{-1}\circ p^{-1}$$
This allows then to apply Lemma~\ref{l1_2017_02_20}, with the measure $\nu:=\mu_{\widehat\phi}\circ\pi^{-1}$, in order to conlude that the coding $p=\pi_\cS:E_A^+ \lra X$ is $\mu_\phi\circ p_1^{-1}$--injective. 

Hence Theorem~\ref{codedskews} applies to give the exact dimensionality of the conditional measure $\mu_\phi^x$, together with the formula for its Hausdorff dimension.

\epf

Now consider the following situation. Let $\cS$, $f$, $F$, and $Y$ be as above. Let 
$
\th: J_S \lra \R
$
be an arbitrary potential such that $\th\circ\pi_\cS:E_A^+\lra \R$ is a locally H\"older continuous summable potential. Let 
\beq\label{5_2017_03_06}
\phi_\th:= \th\circ p_1:J(J_\cS) \lra \R.
\eeq
Then we have:

\blem\label{l1_2017_02_17}
The potential
$
\widehat\phi_\th=\phi_\th\circ p_J\circ \hat \pi: E_A \lra \R
$
is locally H\"older continuous and summable.
\elem

\bpf
Since $
\widehat\phi_\th=(\th\circ\pi_\cS)\circ\pi_0,$ it follows that $\widehat\phi_\th$
is locally H\"older continuous as a composition of two locally H\"older continuous functions. From the definition of summability, the function $\widehat\phi_\th:E_A \to \R$ is summable, since $\th\circ\pi_\cS:E_A^+\to \R$ is summable.
\epf

In this setting, as a result related to Theorem~\ref{EMRskews} (compare also with the case of conditional measures on stable manifolds of endomorphisms in \cite{M-stable}), we get 

\bthm\label{Attracting Smale 2}
Let $\cS$ be a finitely irreducible conformal GDMS. Let  $f:J_\cS\lra J_\cS$ be the corresponding GDMS map. Let $Y \subset \R^d$ be an open bounded set, and let 
$$
F:J_\cS\times Y \lra J_\cS\times Y
$$ 
be a H\"older conformal skew product endomorphism over $f$. 
Let also 
$
\th: J_S \lra \R
$
be a potential such that $\th\circ\pi_\cS:E_A^+\lra \R$ is a locally H\"older continuous summable. Then, 
\sp\begin{itemize}
\item[(a)] For $\mu_{\th\circ\pi_\cS}\circ\pi_\cS^{-1}$--a.e. $x \in J_{\cS}$, the conditional measure $\mu_{\phi_\th}^x$ is exact dimensional on $J_x$, and
$$
\lim_{r\to 0} \frac{\log \mu_{\phi_\th}^x(B(y, r))}{\log r} = \frac{{\h}_{\mu_{\phi_\th}}(F)}{\chi_{\mu_{\phi_\th}}(F)}
$$
for $\mu_{\phi_\th}^x$--a.e. $y \in J_x$; hence, equivalently for $\mu_{\phi_\th}$--a.e $(x,y) \in J(J_\cS)$. 
\item[(b)] The equilibrium state $\mu_{\phi_\th}$ of $\phi_\th:J(J_\cS) \lra \R$ for $F$, is exact dimensional on $J(J_\cS)$ and 
$$
\HD\(\mu_{\phi_\th}\)
=\frac{{\h}_{\mu_{\phi_\th}}(F)}{\chi_{\mu_{\phi_\th}}(F)} +\HD\(\mu_{\th\circ\pi_\cS}\circ\pi_\cS^{-1}\)
=\frac{{\h}_{\mu_{\phi_\th}}(F)}{\chi_{\mu_{\phi_\th}}(F)}+
\frac{{\h}_{\mu_{\th\circ\pi_\cS}}}{{\chi}_{\mu_{\th\circ\pi_\cS}}}
$$
\end{itemize}
\ethm

\begin{proof}
It follows from \eqref{1_2017_01_13} that
$$
\mu_{\phi_\th}\circ p_1^{-1}
=\mu_{\phi_\th\circ p_J\circ\hat\pi}\circ (p_J\circ\hat\pi)^{-1}\circ p_1^{-1}
=\mu_{\th\circ p_1\circ p_J\circ\hat\pi}\circ(p_1\circ p_J\circ\hat\pi)^{-1}
$$
Since $\pi_\cS\circ\pi_0=p_1\circ p_J\circ\hat\pi$, then,
$$
\mu_{\th\circ\pi_\cS}\circ\pi_\cS^{-1}
=\mu_{\th\circ\pi_\cS\circ\pi_0}\circ\pi_0^{-1}\circ\pi_\cS^{-1}
=\mu_{\th\circ p_1\circ p_J\circ\hat\pi}\circ(p_1\circ p_J\circ\hat\pi)^{-1}
$$
Therefore,
$$
\mu_{\phi_\th}\circ p_1^{-1}=\mu_{\th\circ\pi_\cS}\circ\pi_\cS^{-1}.
$$
Hence, (a) follows now directly from Theorem~\ref{EMRskews}, while (b) follows from (a) and Theorem~\ref{globalexact} since exact dimensionality of the measure $\mu_{\th\circ\pi_\cS}\circ\pi_\cS^{-1}$ has been proved in 
\cite{MU-Adv}. Indeed in \cite{MU-Adv} we proved, as a particular case of the random case, the exact dimensionality for all projections of ergodic invariant measures on limit sets of countable conformal IFSs with arbitrary overlaps; and this result extends easily to GDMS. 

\end{proof} 

 An immediate consequence of Theorem \ref{EMRskews} is this:
\begin{cor}\label{EMRskews_B}
Let $f:J_\cS\to J_\cS$ be a finitely irreducible conformal GDMS map, let $Y \subset \R^d$ be an open bounded set, and let 
$$
F:J_\cS \times Y \lra J_\cS \times Y
$$ 
be a H\"older conformal skew product endomorphism over $f$. Let $\phi:J(J_\cS) \lra \R$ be a locally H\"older continuous potential such that 
$
\widehat\phi:= \phi \circ p_J\circ \hat \pi: E_A \lra \R
$
is summable. Then, for $\mu_\phi\circ p_1^{-1}$--a.e $x \in X$, the conditional measure $\mu_\phi^x$ is exact dimensional on $J_x$, and
$$
\lim_{r\to 0} \frac{\log \mu_\phi^x(B(y, r))}{\log r} 
= \frac{\h_{\mu_\phi}(F)}{\chi_{\mu_\phi}(F)}
=\HD(\mu_\phi^x),
$$
for $\mu_\phi^x$--a.e $y \in J_x$; hence, equivalently, for $\mu_\phi$--a.e $(x, y) \in J(J_\cS)$.
\end{cor}

 A Corollary of Theorem~\ref{Attracting Smale 2}, which will be applied to EMR maps (in the sense of \cite{PoW}), is then:

\bcor\label{Attracting Smale 2_B}
Let $\cS$ be a finitely irreducible conformal GDMS. Let $f:J_\cS\to J_\cS$ be the corresponding GDMS map. Let $Y \subset \R^d$ be an open bounded set, and let 
$$
F:J_\cS\times Y \lra J_\cS\times Y
$$ 
be a H\"older conformal skew product endomorphism over $f$. 
Let $\th: J_S \to \R$ be a locally H\"older continuous potential so that 
$
\th\circ\pi_\cS:E_A^+\to \R
$
is a summable potential. Then, 
\begin{itemize}
\item[(a)] For $\mu_{\th\circ\pi_\cS}\circ\pi_\cS^{-1}$--a.e. $x \in J_{\cS}$, the conditional measure $\mu_\phi^x$ is exact dimensional on $J_x$; in fact
$$
\lim_{r\to 0} \frac{\log \mu_{\phi_\th}^x(B(y, r))}{\log r} = \frac{{\h}_{\mu_{\phi_\th}}(F)}{\chi_{\mu_{\phi_\th}}(F)}
$$
for $\mu_{\phi_\th}^x$--a.e. $y \in J_x$; hence, equivalently for $\mu_{\phi_\th}$--a.e $(x,y) \in J(J_\cS)$.

\item[(b)] The equilibrium state $\mu_{\phi_\th}$ of $\phi_\th:J(J_\cS) \lra \R$ for $F$, is exact dimensional on $J(J_\cS)$ and 
$$
\HD\(\mu_{\phi_\th}\)
=\frac{{\h}_{\mu_{\phi_\th}}(F)}{\chi_{\mu_{\phi_\th}}(F)} +\HD\(\mu_{\th\circ\pi_\cS}\circ\pi_\cS^{-1}\)
=\frac{{\h}_{\mu_{\phi_\th}}(F)}{\chi_{\mu_{\phi_\th}}(F)}+
\frac{{\h}_{\mu_{\th\circ\pi_\cS}}}{{\chi}_{\mu_{\th\circ\pi_\cS}}}
$$
\end{itemize}
\ecor

\brem\label{r1_2017_02_23}
As mentioned above, it folows from the last Corollary that our results hold in a particular setting when the derived map $f_\cS:I\lra I$, associated to the GDMS $\cS$, is an expanding Markov-R\'enyi (EMR) map in the sense of \cite{PoW}.
\erem 

\sp Now, consider further an arbitrary conformal GDMS
\[
\cS = \big\{\phi_e:X_{t(e)}\lra X_{i(e)}\big\}_{e\in E}
\]
 Let $\th:J_\cS \lra \R$ be a potential such that $\th\circ\pi_\cS:E_A^+\lra\R$ is locally H\"older continuous and summable. Of particular importance are then the potentials $\th_{q,t}: J_\cS \lra \R$, $t, q\in\R$, 
\beq\label{5_2017_02_23}
\th_{q,t}(\phi_e(x))
:=t\log|\phi_e'(x)|+q(\th(\phi_e(x))-\P(\th))
\eeq
playing a significant role for developing a multifractal analysis of equilibrium states (see also \cite{PoW}). We have then,
$$
\th_{q,t}\circ\pi_\cS(\om)
=t\log\big|\phi_{\om_0}'(\pi_\cS(\sg(\om)))\big|
+q\(\th\circ\pi_\cS(\om)-\P(\phi)\)
$$
In terms of the GDMS map $f_\cS:J_\cS\lra J_\cS$ associated to $\cS$, and defined by \eqref{1_2016_12_16}, we have
\beq\label{5_2017_02_23_B}
\th_{q,t}(x)
=-t\log|f'(x)|+q(\th(x)-\P(\th)).
\eeq
From the Bounded Distortion Property (BDP) from the definition of conformal GDMSs, the first summand in the above formula is H\"older continuous and the second one is H\"older continuous by its definition. Thus, we obtain:

\blem\label{l2_2017_02_23}
For all $q, t\in\R$ the potential $\th_{q,t}\circ\pi_\cS:E_A^+\lra$ is locally H\"older continuous.
\elem

 The problem of parameters $q, t\in\R$ for which $\th_{q,t}\circ\pi_\cS:E_A^+\to\R$ are summable is more delicate and has been treated in detail in \cite{HMU}. 
The importance of the geometric potentials $\th_{q,t}\circ\pi_\cS$ is due to the fact that these are suitable for a description of geometry of the limit set $J_\cS$. If $q=0$, then the parameter $t\ge 0 $ for which $\P(\th_{0,t})=0$ (if it exists) coincides with the Hausdorff dimension $\HD(J_\cS)$ of the limit set $J_\cS$, and  also these potentials play a role in multifractal analysis of the equilibrium state $\mu_{\th\circ\pi_\cS}\circ\pi_\cS^{-1}$; for example as in \cite{HMU}, \cite{gdms}, \cite{PoW}. 

Denote by $\Sg(\cS,\th)$ the set of pairs $(t,q)\in\R^2$ for which the potential $\th_{q,t}\circ\pi_\cS:E_A^+\to\R$ is summable and by $\Sg_0(\cS,\th)$ the set of those $q\in\R$ for which there exists a (unique) real number $T(q)$, usually refered to as a temperature, so that $(q,T(q))\in \Sg(\cS,\th)$ and 
$$
\P\(\th_{q,T(q)}\circ\pi_\cS\)=0
$$
We now are in the setting of Lemma~\ref{l1_2017_02_17} and Theorem~\ref{Attracting Smale 2}. For $q\in \Sg_0(\cS,\th)$ abbreviate
$$
\psi_q:=\phi_{\th_{q,T(q)}}=\th_{q,T(q)}\circ p_1: J(J_\cS)\lra\R
$$
From Theorem~\ref{Attracting Smale 2}, we get the following.

\bcor\label{c3_2017_0223}
With $\th_{q,t}$ defined in \eqref{5_2017_02_23} and with notation following it, we have the following. If $q\in \Sg_0(\cS,\th)$, then,

\fr(a) For $\mu_{\th_{q,T(q)}\circ\pi_\cS}\circ\pi_\cS^{-1}$--a.e. $x \in J_{\cS}$, the conditional measure $\mu_{\psi_q}^x$ is exact dimensional on $J_x$; in fact
$$
\lim_{r\to 0} \frac{\log \mu_{\psi_q}^x(B(y, r))}{\log r} = \frac{{\h}_{\mu_{\psi_q}}(F)}{\chi_{\mu_{\psi_q}}(F)}
$$
for $\mu_{\psi_q}^x$--a.e. $y \in J_x$; hence, equivalently for $\mu_{\psi_q}$--a.e $(x,y) \in J(J_\cS)$.

\fr(b) The equilibrium state $\mu_{\psi_q}$ of $\psi_q:J(J_\cS)\to \R$ for $F$, is exact dimensional on $J(J_\cS)$, and
$$
\HD\(\mu_{\psi_q}\)
=\frac{{\h}_{\mu_{\psi_q}}(F)}{\chi_{\mu_{\psi_q}}(F)} +\HD\(\mu_{\th_{q,T(q)}\circ\pi_\cS}\circ\pi_\cS^{-1}\)
=\frac{{\h}_{\mu_{\phi_\th}}(F)}{\chi_{\mu_{\phi_\th}}(F)}+
\frac{{\h}_{\mu_{\th_{q,T(q)}\circ\pi_\cS}}}{{\chi}_{\mu_{\th_{q,T(q)}\circ\pi_\cS}}}
$$ 
\ecor

\

\subsection{Skew products with conformal parabolic GDMSs in the base} \label{S-parabolicgdms}
Now we pass to the second large class of examples, built on parabolic iterated function systems. 

Assume again that we are given 

\begin{itemize}
\item a directed multigraph $(V,E,i,t)$ ($E$ countable, $V$ finite), 

\item an incidence matrix $A:E\times E\lra \{0,1\}$, 

\item two functions $i, t:E\lra V$ such that $A_{ab} = 1$ implies $t(a) = i(b)$. 

\item  nonempty compact metric spaces $\{X_v\}_{v\in V}$. 
\end{itemize}

Suppose further that we have  a collection of
conformal maps 
$$
\phi_e:X_{t(e)}\lra X_{i(e)},  \  \  e\in E, 
$$
satisfying the following conditions (which are more general than the above in that we do not necessarily assume that the maps are uniform contractions):

\begin{itemize}
\item[(1)](Open Set Condition) \ 
$
\phi_a(\Int(X))\cap \phi_b(\Int(X))=\es
$, 
for all $a, b\in E$ with $a\ne b$.

\item[(2)] $|\phi_e'(x)|<1$ everywhere except for finitely many
pairs $(e,x_e)$, $e\in E$, for which $x_e$ is the unique fixed point
of $\phi_e$ and $|\phi_e'(x_e)|
=1$. Such pairs and indices $e$ will be called \textit{parabolic} and the set of
parabolic indices will be denoted by $\Om$. All other indices will be 
called \textit{hyperbolic}. We assume $A_{ee}=1$ for all $e\in\Om$.

\item[(3)]  $\forall n\ge 1 \  \forall \om = (\om_1\om_2...\om_n)\in E_A^n$
if $\om_n$ is a hyperbolic
index or $\om_{n-1}\ne \om_n$, then $\phi_{\om}$ extends conformally to
an open connected $W_{t(\om_n)}\sbt\R^d$ and maps $W_{t(\om_n)}$ into $W_{i(\om_n)}$.

\item[(4)] If $e\in E$ is a parabolic index, then 
$$
\bi_{n\ge 0}\phi_{e^n}(X)
=\{x_e\}
$$ 
and the diameters of the sets $\phi_{e^n}(X)$ converge
to 0.

\item[(5)] (Bounded Distortion Property) $\exists K\ge 1 \  \forall
n\ge 1
 \  \forall \om\in E_A^n  \  \forall x,y\in W_{t(\om_n)}$, if $\om_n$ 
is a hyperbolic index or $\om_{n-1}\ne \om_n$, then
$$
{|\phi_\om'(y)| \over |\phi_\om'(x)| } \le K
$$
\item[(6)] $\exists \ka<1 \  \forall n\ge 1  \  \forall \om\in E_A^n$ if 
$\om_n$ is a hyperbolic index or $\om_{n-1}\ne \om_n$, then
$\|\phi_\om'\|\le \ka$.

\item[(7)] (Cone Condition)  There exist $\a,l>0$ such that for
 every $x\in\bd X \sbt\R^d$ there exists an open cone 
$\Con(x,\a,l)\sbt \Int(X)$ with vertex $x$, central
angle of Lebesgue measure $\a$, and altitude $l$.

\item[(8)] There exists a constant $L\ge 1$ such that for every $e\in E$ and every $x,y\in V$,
\[
\bigg|\frac{|\phi_e'(y)|}{|\phi_e'(x)|}-1 \bigg| \le L\|y-x\|^\alpha,
\]
\end{itemize}

\fr We call such a system 
$
\cS=\{\phi_e:e\in E\}
$ 
a {\it  subparabolic conformal graph directed Markov system. }

\bdfn\label{Def_Parabolic}
If $\Om\ne\es$, we call the system $\cS=\{\phi_i:i\in E\}$ {\it parabolic}. 
\edfn

As declared in (2) the elements of the set $E\sms \Om$ are called
hyperbolic.
We extend this name to all the words appearing in (5) and (6). It follows
from (3) that for every hyperbolic word $\om$, \
$
\phi_\om(W_{t(\om)})\sbt W_{t(\om)}.
$
Note that our conditions ensure that $\f_e'(x) \neq 0$ 
for all $e\in E$ and all $x \in X_{t(i)}$. It was proved (though only for IFSs but the case of GDMSs can be treated completely similarly) in \cite{MU_Parabolic_1} (comp. \cite{gdms}) that
\beq\label{1_2016_03_15}
\lim_{n\to\infty}\sup_{\om\in E_A^n}\big\{\diam(\phi_\om(X_{t(\om)}))\big\}=0.
\eeq
This implies then:

\bcor\lab{p1c2.3} 
The map $\pi:E_A^+\lra X:=\du_{v\in V}X_v$, \ 
$
\{\pi(\om)\}:=\bi_{n\ge 0}\phi_{\om|_n}(X),
$
is well--defined,  and  $\pi$ is uniformly continuous.
\ecor

\fr Similarly as for hyperbolic (attracting) systems the limit set $J = J_\cS$ of $\cS = \{\f_e\}_{e\in e}$ is,
$$
J_\cS:=\pi(E_A^+)
$$
and satisfies the following self-reproducing property:
$
J = \bu_{e\in E} \f_e(J).
$

We now want to associate to the parabolic system $\cS$ a canonical hyperbolic system $\cS^*$; this will be done by using the jump transform (\cite{Sch}). We will then be able to apply the ideas from the previous section 
to $\cS^*$. The set of edges is:
$$
E_*:= \big\{i^nj: n\ge 1, \  i\in\Om, \ i\ne j\in E, \ A_{ij}=1\big\} \cup 
(E\sms \Om)\sbt E_A^*
$$ 
We set
$$
V_*:=V
$$
and keep the functions $t$ and $i$ on $E_*$ as the restrictions of $t$ and $i$ from $E_A^*$. The incidence matrix $A^*:E_*\times E_*\to\{0,1\}$ is defined in the natural (the only reasonable) way by declaring that $A^*_{ab}=1$ if and only if $ab\in E_A^*$. Finally 
$$
\cS^*:=\big\{\phi_e:X_{t(e)}\lra X_{t(e)}|\, e\in E^*\big\}
$$
It follows from our assumptions (see \cite{MU_Parabolic_1} and \cite{gdms}) that the following is true.

\bthm\lab{p1t5.2} 
The system $\cS^*$ is a hyperbolic (contracting) conformal GDMS and the limit sets $J_\cS$ and $J_{\cS^*}$ differ only by a countable set. If the system $\cS$ is finitely irreducible, then so is the system $\cS^*$.
\ethm

\fr The most important advantage of $\cS^*$ is that it is a an attracting conformal GDMS. On the other hand, the price we pay by replacing  the non--uniform contractions in $\cS$ with the uniform contractions in $\cS^*$ is that even if the alphabet $E$ is finite, the alphabet $E_*$ of $\cS^*$ is always infinite. Thus we will be able to apply our results on infinite Smale systems.  We have the following quantitative behavior around parabolic points.

\bprop\lab{p1c5.13} 
Let $\cS$ be a conformal parabolic GDMS. Then there exists a constant $C\in(0, \infty)$ and for every $i\in\Om$ there exists some constant
$\beta_i\in(0, \infty) $ such that for all $n\ge 1$ and for all $z\in X_i:=
\bu_{j\in I\sms\{i\}}\phi_j(X)$,
$$
C^{-1}n^{-{\beta_i+1\over \beta_i}}\le |\phi_{i^n}'(z)|
\le Cn^{-{\beta_i+1\over \beta_i}}
$$
If  $d=2$ then all constants $\b_i$ are integers $\ge 1$, and if $d\ge 3$ then all  $\b_i$ are equal to $1$. 
\eprop

\fr From Theorem~\ref{Attracting Smale 2} we obtain:

\begin{cor}\label{parabolic Smale} 
Let $\cS$ be an irreducible conformal parabolic GDMS. Let $\cS^*$ be the corresponding atracting conformal GDMS produced in Theorem~\ref{p1t5.2}. Furthermore, let  $f:J_{\cS^*}\to J_{\cS^*}$ be the corresponding GDMS map. Let $Y \subset \R^d$ be an open bounded set, and let 
$$
F:J_{\cS^*}\times Y\lra J_{\cS^*} \times Y
$$ 
be a H\"older conformal skew product endomorphism over $f$.
Let 
$
\th: J_{\cS^*}\to \R
$
be a potential such that $\th\circ\pi_{\cS^*} :E_{A^*}^+\to \R$ is a locally H\"older continuous summable potential. Then, 

\fr(a) For $\mu_{\th\circ\pi_\cS}\circ\pi_{\cS^*}^{-1}$--a.e $x \in J_{\cS^*} $, the conditional measure $\mu_{\phi_\th}^x$ is exact dimensional on $J_x$ and for $\mu_{\phi_\th}^x$--a.e. $y \in J_x$ (hence, equivalently for $\mu_{\phi_\th}$--a.e $(x,y) \in J(J_{\cS^*} )$) we have,
$$
\lim_{r\to 0} \frac{\log \mu_{\phi_\th}^x(B(y, r))}{\log r} = \frac{{\h}_{\mu_{\phi_\th}}(F)}{\chi_{\mu_{\phi_\th}}(F)}
$$
\fr (b) The equilibrium state $\mu_{\phi_\th}$ of $\phi_\th:J(J_\cS^*) \lra \R$ for $F$, is exact dimensional on $J(J_\cS^*)$ and 
$$
\HD\(\mu_{\phi_\th}\)
=\frac{{\h}_{\mu_{\phi_\th}}(F)}{\chi_{\mu_{\phi_\th}}(F)} +\HD\(\mu_{\th\circ\pi_{\cS^*}}\circ\pi_{\cS^*}^{-1}\)
=\frac{{\h}_{\mu_{\phi_\th}}(F)}{\chi_{\mu_{\phi_\th}}(F)}+
\frac{{\h}_{\mu_{\th\circ\pi_{\cS^*}}}}{{\chi}_{\mu_{\th\circ\pi_{\cS^*}}}}
$$
\end{cor}

\fr From Corollary~\ref{c3_2017_0223}, by replacing $\cS$ with $\cS^*$, we get: 

\bcor\label{c3_2017_0223_parabolic}
With $\th_{q,t}$ defined in \eqref{5_2017_02_23} and with notation following it, we have the following. If $q\in \Sg_0(\cS^*,\th)$, then 
we obtain for the potential $\psi_q = \theta_{q, T(q)} \circ p_1: J(J_{\cS^*})\lra \mathbb R$ the same conclusions (a), (b) as in Corollary \ref{parabolic Smale}.

\ecor

 We investigate now an example, denoted by $\mathcal I$, formed by the inverse maps of the two continuous pieces of 
the Manneville--Pomeau map $f_\a:[0, 1]\lra [0, 1]$ defined by: 
$$
f_\a(x) = x + x^{1+\alpha} \ (\text{mod}\,1),
$$
where $\alpha>0$ is arbitrary fixed. Of course the GDMS map resulting from $\cI$ is $f_\a$. 

\sp\fr Then as a  consequence of Corollary~\ref{parabolic Smale}, we obtain:

\bcor\label{Attracting Smale 2 MP}
$\alpha>0$. Let $Y \subset \R^d$ be an open bounded set, and let 
$$
F_\a:J_\cI\times Y \lra J_\cI\times Y
$$ 
be a H\"older conformal skew product endomorphism over the Manneville--Pomeau map $f_\a$. 

Let $\th: J_\cI\lra \R$ be a potential such that $\th\circ\pi_\cI:\N^+\lra \R$ is a locally H\"older continuous summable potential. Then, 
the conclusions of Corollary \ref{parabolic Smale} will hold in this case.


\ecor

Also for the geometric potentials $\theta_{q, T(q)}$ we obtain:

\bcor\label{c3_2017_0223_MP}
With $\th_{q,t}$ defined in \eqref{5_2017_02_23}, then for $q\in \Sg_0(\cI,\th)$, we obtain the same conclusions (a), (b) as in Corollary \ref{Attracting Smale 2 MP} for the potential $\theta_{q, T(q)}$.



\ecor

\brem\label{r1_2016_12_15}
In constructing the attracting conformal GDMS of Theorem~\ref{p1t5.2} we built on Schweiger's jump transformation, \cite{Sch}. We could also use inducing on each  $X_v$, $v\in V\sms \Om$, considering the system generated by  $\phi_\om$ where $i(\om_1)=t\(\om|_{|\om|}\)=v$ and $i(\om_k)\ne v$ for all $k=2, 3,\ld,|\om|-1$. The ``jump'' construction of Theorem~\ref{p1t5.2} seems to be somewhat better as it usually leads to a smaller system.
\erem

\subsection{Backward continued fractions}\label{backwrd}
\ 
 Each irrational $x \in [0, 1]$ has a unique expansion in the form of backward continued fraction:
$$
x = \frac{1}{\om_1-\frac{1}{\om_2-\frac{1}{\om_3-\frac{1}{\ldots}}}},
$$
where $\om_1, \om_2, \ldots$ are integers $\ge 2$. The corresponding map to the Gauss map from the regular continued fractions, is the R\'enyi map $V
: [0, 1) \lra [0, 1)
$
given by the formula
$$
V(x):= 
\begin{cases}
\lt \{\frac{1}{1-x}\rt\}
 &{\rm if} \ x \ne 0, \\ 
  0  &{\rm if} \ x = 0
\end{cases}
$$
The graph of R\'enyi map is the reflection of the graph of the Gauss map in the line $x = \frac 12$. 

The backward continued fraction system $\mathcal S$ is given by the maps:
\begin{equation}\label{backwardf}
\phi_i: [0, 1] \to [0, 1], \  \  \phi_i(x) = \frac{1}{i-x}, \  i \ge 2.
\end{equation}
Hence the map $\phi_2$ has a neutral fixed point at 1, and is contracting everywhere else.   All the other maps $\phi_j$ are contracting on $[0, 1]$, for $j \ge 3$. \ 
Unlike for the regular continued fractions where the invariant absolutely continuous measure for the Gauss map is the measure $\frac{dx}{1+x}$,  \ for backward continued fractions, the invariant absolutely continuous measure for $V$ is $\frac{dx}{x}$ (see \cite{AF}).  
The backward continued fraction map $V$ is a factor of a cross--section map for the geodesic flow on the unit tangent bundle of the modular surface. 

The system $\mathcal S$ given by (\ref{backwardf}) satisfies our conditions for a parabolic system, and we can associate to it the contracting system $\mathcal S^*$ (using the jump transformation).
We can then apply Corollary \ref{parabolic Smale} with regard to a 
 H\"older conformal skew product endomorphism $F: J_{\cS^*}\times Y\lra J_{\cS^*} \times Y$ over $f$, to obtain the exact dimensionality of the conditional measures of the equilibrium measure $\mu_\phi$ of any H\"older continuous summable potential $\phi$. We infer also the global exact dimensionality of $\mu_\phi$ on $  J_{\cS^*}\times Y$.

\

\section{Natural Extensions of Graph Directed Markov Systems}\label{Natural Extensions of GDMS}

Finally in this section, we want to define and explore the systems we refer to which as natural extensions of graph directed Markov systems. More precisely, let 
$
\cS = \big\{\phi_e:X_{t(e)}\to X_{i(e)}\big\}_{e\in E}
$
be a conformal finitely irreducible graph directed Markov system and let 
$$
f=f_\cS:J_\cS\lra J_\cS
$$
be the GDMS map associated to the system $\cS$ and given by formula \eqref{1_2016_12_16}. Fix  an arbitrarily chosen point $\xi\in J_\cS$. We then define the skew product map 
$$
\^f:J_\cS\times \ov J_\cS\lra J_\cS\times \ov J_\cS,
$$ 
\beq\label{2_2017_03_03}
\^f(\phi_e(x),y):=(x,\phi_e(y))=\big(f(\phi_e(x)),\phi_e(y)\big),
\eeq
if $x\in\Int(X_{t(e)})$ (then $e$ is uniquely determined), and 
\beq\label{A2_2017_03_03}
\^f(z,y):=(\xi,\xi),
\eeq
if $z\notin \bu_{e\in E} \phi_e\(\Int(X_{t(e)})\)$. We call it the \textit{natural extension} (or \textit{inverse limit}) of  $f$, see also \ref{generalskew}. For applications to endomorphisms see for eg \cite{Ru2}, \cite{M-approx}, \cite{M-ETDS11}, \cite{M-MZ}, \cite{M-DCDS12}. In \cite{M-MZ} it was constructed and studied in detail a large class of skew product endomorphisms with overlaps in fibers, which dynamically are far from both homeomorphisms and constant-to-1 endomorphisms, in regard to their fiber dimensions.
From Theorem~\ref{EMRskews} we obtain:

\bcor\label{c1_2016_12_16c1_2016_12_16}
Let $\cS$ be a conformal finitely irreducible graph directed Markov system and let $f=f_\cS:J_\cS\to J_\cS$ be the corresponding GDMS map. Let also
$$
\^f:J_\cS\times \ov J_\cS\lra J_\cS\times \ov J_\cS
$$ 
be the natural extension of the map $f$, as defined above. Let $\phi:J_\cS\times \ov J_\cS\to\R$ be a potential such that 
$
\widehat\phi:= \phi \circ p_J\circ \hat \pi: E_A \lra \R
$
is locally H\"older continuous and summable. 
Then, for $\mu_\phi\circ p_1^{-1}$--a.e $x \in J_\cS$, the conditional measure $\mu_\phi^x$ is exact dimensional on $\ov J_\cS$, and
$$
\lim_{r\to 0} \frac{\log \mu_\phi^x(B(y, r))}{\log r} = \frac{{\h}_{\mu_\phi}(\^f)}{\chi_{\mu_\phi}(\^f)},
$$
for $\mu_\phi^x$--a.e $y \in \ov J_\cS$; hence, equivalently, for $\mu_\phi$--a.e $(x,y) \in J_\cS\times \ov J_\cS $.

\ecor



 As a consequence of Theorem~\ref{Attracting Smale 2} we obtain the following.

\bcor\label{c1_2017_03_03}
Let $\cS$ be a conformal finitely irreducible graph directed Markov system and let $f=f_\cS:J_\cS\to J_\cS$ be the corresponding GDMS map. Furhermore, let 
$$
\^f:J_\cS\times \ov J_\cS\lra J_\cS\times \ov J_\cS
$$
 be the natural extension of the map $f$. Let $\th: J_S \to \R$ be an arbitrary potential such that 
$
\th\circ\pi_\cS:E_A^+\lra J_\cS
$ 
is a locally H\"older continuous summable potential. Then, 
\begin{itemize}
\item[(a)] For $\mu_{\th\circ\pi_\cS}\circ\pi_\cS^{-1}$--a.e. $x \in J_{\cS}$, the conditional measure $\mu_{\phi_\th}^x$ is exact dimensional on $J_x$; in fact for $\mu_{\phi_\th}^x$--a.e. $y \in J_x$,
$$
\lim_{r\to 0} \frac{\log \mu_{\phi_\th}^x(B(y, r))}{\log r} = \frac{{\h}_{\mu_{\phi_\th}}(\^f)}{\chi_{\mu_{\phi_\th}}(\^f)}
$$
\item[(b)] The equilibrium state $\mu_{\phi_\th}$ of $\phi_\th:J_\cS\times \ov J_\cS \lra \R$ for $\^f$, is exact dimensional and,
$$
\HD\(\mu_{\phi_\th}\)
=\frac{{\h}_{\mu_{\phi_\th}}(\^f)}{\chi_{\mu_{\phi_\th}}(\^f)} +\HD\(\mu_{\th\circ\pi_\cS}\circ\pi_\cS^{-1}\)=\frac{{\h}_{\mu_{\phi_\th}}(\^f)}{\chi_{\mu_{\phi_\th}}(\^f)}+\frac{{\h}_{\mu_{\th\circ\pi_\cS}}}{{\chi}_{\mu_{\th\circ\pi_\cS}}}
$$
\end{itemize}
\ecor

 Thus, in the important case of the Gauss map and the continued fraction system, we have the exact dimensionality of the conditional measures on the fibers of the natural extension, and also the global exact dimensionality of equilibrium measures:

\bcor\label{c1_2017_03_03B}
Let $\Q^c$ be the set of all irrational numbers in the unit interval $[0,1]$. Let 
$
G:\Q^c\lra\Q^c
$
the Gauss map
$
G(x):=\lt\{\frac 1x\rt\}, 
$
and let the iterated function system on $[0,1]$ 
$$
\cS_\cG=\lt\{[0,1]\ni x\longmapsto \frac{1}{x+n}\in[0,1]\rt\}_{n\in\N},
$$
consisting of all inverse branches of the Gauss map $G$, so that
$
G=f_{\cS_\cG}.
$
Let 
$$
\^G:J_\cG\times \ov J_\cG\lra J_\cG\times \ov J_\cG
$$ 
be the natural extension of  $G$. Let $\th: J_G \to \R$ be a potential such that 
$
\th\circ\pi_\cG:E_A^+\to \R
$ 
is locally H\"older continuous summable. Then the conclusion of Corollary \ref{c1_2017_03_03} holds.

\ecor

\sp

\textbf{Remark.} We recall that in \cite{MU-ETDS} we used these results to study the approximation coefficients $\Theta_n(x)$ from the continued fraction expansion of $x$, extending and developing the \textit{Doeblin-Lenstra Conjecture}  (see \cite{BJW}, \cite{DK}, \cite{IK}). If $x \in [0, 1)$ is an irrational number with continued fraction expansion 
$$
x = \frac{1}{a_1 +\frac{1}{a_2+\ldots}}, \ a_i \ge 1, i \ge 1,
$$
and if 
$$
\frac{p_n}{q_n} = \frac{1}{a_1+\frac{1}{a_2+ \frac{1}{\ldots +\frac{1}{a_n}}}}
$$ 
is the truncated expansion at order $n$, with integers $p_n, q_n \ge 1$, for $n \ge 1$, then 
$$
\Theta_n(x) := q_n^2\lt|x - \frac{p_n}{q_n}\rt|
$$
The original Doeblin-Lenstra Conjecture gave the distribution of $\Theta_n(x)$ associated to numbers $x$ from a set $A\subset [0, 1)$ of Lebesgue measure 1. We studied the distribution of $\Theta_n(x)$ for numbers $x$ outside $A$, namely from sets in $[0, 1)$ having (singular) $\mu$-equilibrium measure equal to 1. Moreover, we improved the Doeblin-Lenstra conjecture in this case, by describing the frequency of visits of $\Theta_n(x)$ to arbitrarily small intervals.



\

\section{Generalized L\"uroth Systems and Their Inverse Limits}\label{S-GLS}

We include this short section since it is needed for the full treatment of $\b$-transformations, for arbitrary $\b>1$, and of their natural extensions. It just collects the results of the previous sections in the case of a special subclass of maps.

\bdfn
We call an iterated function system 
$
\cS=\{\phi_e:I\to I\}_{e\in E}
$
a L\"uroth system if the maps $\phi_e:I\lra I$, $e\in E$,  are of the form 
$
x\longmapsto ax+b, \  \  \a\ne 0,
$ 
and \ 
$
{\rm Leb}\lt(\bu_{e\in E}\phi_e(I)\rt)=1.
$
\edfn

\fr Let $f=f_\cS:J_\cS\to J_\cS$ be the map associated to the system $\cS$ and given by formula \eqref{1_2016_12_16}:
\beq\label{3_2017_03_16}
f(\phi_e(x))=x
\eeq
Writing 
$$
\phi_e(x)=a_ex+b_e, \  \  e\in E,
$$
with $a_e\in(0,1)$, we rewrite \eqref{3_2017_03_16} in the following more explicit form
\beq\label{A1_2017_03_16}
f(x)=
\begin{cases}
a_e^{-1}x-a_e^{-1}b_e, &{\rm if } \  \  x\in \Int(\phi_e(I)) \\
0, &{\rm if } \  \  x\notin \bu_{e\in E}\Int(\phi_e(I)).
\end{cases}
\eeq
In particular, we took the given preassigned point $\xi$ involved in  \eqref{1_2016_12_16} to be $0$. In the sequel we will need however mainly only the definition of $f$ on $\bu_{e\in E}\Int(\phi_e(I))$. 

\

The \textit{natural extension} $\^f:J_\cS\times \ov J_\cS\to J_\cS\times \ov J_\cS$ of $f$ is given by formula \eqref{2_2017_03_03}. In more explicit terms, we have:
\begin{equation}\label{natextGLS}
\^f(x,y) = 
\begin{cases}
(f(x),a_ey+b_e)=(a_e^{-1}x-a_e^{-1}b_e,a_ey+b_e\),  &x\in \Int(\phi_e(I)) \\
(0,0),  &x\notin \bu_{e\in E}\Int(\phi_e(I))
\end{cases}
\end{equation}
As consequences  of Corollary~\ref{c1_2016_12_16c1_2016_12_16},
we have:

\bcor\label{c1_2016_12_16X}
Let $\cS$ be a L\"uroth system system and let $f=f_\cS:J_\cS\to J_\cS$ be the corresponding GDMS map. Furhermore, let
$$
\^f:J_\cS\times \ov J_\cS\lra J_\cS\times \ov J_\cS
$$ 
be the natural extension of the map $f$. Let $\phi:J_\cS\times \ov J_\cS\to\R$ be a potential such that 
$
\widehat\phi= \phi \circ p_J\circ \hat \pi: E_A \lra \R
$
is locally H\"older continuous and summable. 
Then, for $\mu_\phi\circ p_1^{-1}$--a.e $x \in J_\cS$, the conditional measure $\mu_\phi^x$ is exact dimensional on $\ov J_\cS$, and moreover,
$$
\lim_{r\to 0} \frac{\log \mu_\phi^x(B(y, r))}{\log r} = \frac{{\h}_{\mu_\phi}(\^f)}{\chi_{\mu_\phi}(\^f)},
$$
for $\mu_\phi^x$--a.e $y \in \ov J_\cS$; hence, equivalently, for $\mu_\phi$-a.e $(x,y) \in J_\cS\times \ov J_\cS $.

In particular, the above conclusions hold if $\phi:J_\cS\times \ov J_\cS\to\R$ is a locally H\"older continuous potential with
$
\widehat\phi= \phi \circ p_J\circ \hat \pi: E_A \to \R
$ summable.  
And the above conclusions hold if $\th: J_S \to \R$ with $\th\circ\pi_\cS:E_A^+\to J_\cS$ locally H\"older continuous summable potential. 
\ecor

\

\section{Thermodynamic Formalism for Inverse Limits of $\beta$--Maps,  $\beta > 1$}\label{beta-t}

For arbitrary $\beta>1$, the $\beta$--transformation is $T_\beta: [0, 1) \to [0, 1)$, 
$$
T_\beta(x) = \beta x  \ \text{(mod 1)},
$$ 
and the $\beta$--expansion of $x$ is:
$$
x = \mathop{\sum}\limits_{k=1}^\infty \frac{d_k}{\beta^k},
$$
where 
$
d_k = d_k(x) = [\beta T_\beta^{k-1}(x)], \ k \ge 1,
$
where as usually $[a]$ denotes the integer part of a real number $a$. The \textit{digits} $d_k, \ k \ge 1$ are chosen from  the finite set $\{0, 1, \ldots, [\beta]\}$.
\newline
Not all infinite series of the form 
$
\mathop{\sum}\limits_{k\ge 1} \frac{d_k}{\beta^k},
$
with $d_k \in \{0, 1, \ldots, [\beta]\}$ are however $\beta$--expansions of some number. We say that a sequence of digits $(d_1, d_2, \ldots)$ is \textit{admissible} if there is  $x \in [0, 1)$ with $\beta$-expansion  
$$
x = \mathop{\sum}\limits_{k \ge 1} \frac{d_k}{\beta^k}.
$$
The map  $T_\beta$ does not necessarily preserve Lebesgue measure $\lambda$, however it has a unique probability measure $\nu_\beta =  h_\beta d\lambda$, which is equivalent to $\lambda$ and $T_\beta$--invariant. Its density $h_\beta$ has an explicit form (see \cite{Pa}, \cite{DK}). 

The dynamical and metric properties of the transformation $T_\beta$ have been studied by many authors, for eg \cite{DKa}, \cite{DK}, \cite{DKS}, \cite{Pa}. Also several authors studied the $T_\beta$-invariant set $U_\beta$ of real numbers $x$ which have a unique $\beta$-representation (also called \textit{univoque numbers} for $\beta$-expansion), for eg \cite{AK}, \cite{GS}. In this case there is a bijection between $U_\beta$ and the set $\mathcal U_\beta$ of uniquely defined prehistories of points from $U_\beta$. 

\

 Consider now the inverse limit of the system $([0, 1), T_\beta)$. First, let us take  $\beta = \frac{\sqrt 5 +1}{2}$, as a simpler example. Define the following skew product map
$$
\mathcal T_\beta(x, y) 
:= \lt(T_\beta(x), \frac{y+[\beta x]}{\beta}\rt),
$$
on a subset $Z \subset [0, 1)^2$, where the horizontal $[0, 1)$  is considered as the future-axis and the vertical $[0, 1)$ is considered as the past-axis. The inverse limit of $T_\beta$ must encapsulate both the forward iterates of $T_\beta$, and the backward trajectories of $T_\beta$. For $\beta = \frac{1+\sqrt 5}{2}$, take 
$$
Z = [0,1/\beta) \times [0, 1) \ \bigcup \ [1/\beta, 1) \times [0,1/\beta)
$$
Then the map $\mathcal T_{\beta}: Z \to Z$ is well defined and bijective, and it is the inverse limit of $T_\beta$.
 For $\beta = \frac{\sqrt 5 +1}{2}$, the set of admissible sequences forms a subshift of finite type $E_2^+(11)$, defined as the set of sequences in $E_2^+$ which do not contain the forbidden word $11$.
In this case there is a H\"older continuous coding map $\pi:E_2^+(11) \lra [0, 1)$, given by
$$
\pi\big((d_1, d_2, \ldots) \big) = \sum\limits_{n \ge 1} \frac{d_n}{\beta^n}
$$
We can then take skew product endomorphisms over $T_{\frac{1+\sqrt 5}{2}}$, and obtain the following.

\begin{thm}\label{sqrt}
Let $F:[0, 1) \times Y \lra [0, 1)\times Y$ be a skew product endomorphism given by
$$
F(x, y) = \lt(\frac{1+\sqrt 5}{2} x \ ({\rm mod}\, 1), g(x, y)\rt),
$$
where $Y$ is an open bounded set in $\R^d$ and  $g(x, \cdot): Y \lra Y$ is a conformal map for every $x \in [0, 1)$. Assume also that for any $x_{-1}, x_{-1}' \in [0, 1)$ with
$
T_{\frac{1+\sqrt 5}{2}}x_{-1} = T_{\frac{1+\sqrt 5}{2}}x'_{-1},
$
we have 
$$
g(x_{-1}, Y) \cap g_{x_{-1}'}(Y) = \emptyset,
$$
and that the fiber maps $g_x$ satisfy conditions (a)--(g) of subsection \ref{CSPSE}. Then, for any locally H\"older continuous potential 
$
\phi: [0, 1) \times Y \lra \R,
$
its equilibrium measure $\mu_\phi$ has exact dimensional conditional measures $\mu_\phi^x$, for $\mu_\phi\circ\pi_0^{-1}$--a.e $x \in [0, 1)$.
\end{thm}

\begin{proof}
This case is simpler because the coding space  is a subshift of finite type on finitely many symbols. Moreover we see that the coding map $\pi: E_2^+(11)\to [0, 1)$ over $T_{\frac{1+\sqrt 5}{2}}$ is injective outside a countable set, and  the associated symbolic skew product 
$$
\hat F: E_2^+(11)\times Y \lra E_2^+(11) \times Y
$$ 
over $E_2^+(11)$, satisfies conditions (a)--(g) in subsection \ref{CSPSE}.
Therefore if $\phi$ is locally H\"older we have also the summability condition, since we work with finitely many symbols, thus Theorem \ref{codedskews} gives the conclusion.

\end{proof}

A similar case is for $\beta$ a \textit{pseudo-golden mean number of order} $m$ (or multinacci number), i.e $\beta>1$ is the positive root of the polynomial  $z^m-z^{m-1} - \ldots -1$. Thus 1 has a finite $\beta$-expansion, 
$$1 = \frac{1}{\beta} + \ldots + \frac{1}{\beta^m}$$
In this case the inverse limit homeomorphism $\mathcal T_\beta: Z \to Z$  is defined on the finite union
$$Z = \mathop{\bigcup}\limits_{j=0}^{m-1} \Big[T_\beta^{m-j}(1), T_\beta^{m-j-1}(1)\Big) \times \Big[0, T_\beta^j(1)\Big)$$
Then we can construct a Smale skew product as above for $\mathcal T_\beta$, and prove the exact dimensionality of equilibrium measures, and the formula for their dimension. So a similar result as Theorem \ref{sqrt} is obtained for these $\beta$'s.
The partition of $[0, 1)$ given by the intervals 
\begin{equation}\label{intS}
\lt[0, \frac 1\beta\rt), \lt[\frac 1\beta, \frac 1\beta + \frac {1}{\beta^2}\rt), \ldots, \lt[\frac{1}{\beta} + \ldots + \frac{1}{\beta^{m-1}}, 1\rt),
\end{equation}
gives a GLS map $S$, which is piecewise affine and takes any such interval onto $[0, 1)$. Then, the natural extension $\mathcal S : [0, 1)^2 \to [0, 1)^2$ of $S$ has the intervals from (\ref{intS}) on its image on the vertical coordinate (see \cite{DK}), and is given by: 

\begin{equation}\label{mathcalSm}
\mathcal S(x, y)  
= 
\begin{cases}
   \lt(\beta x, \frac y\beta\rt), & {\rm \   if} \  (x, y) \in \lt[0, \frac 1\beta\rt) \times [0, 1), \vspace{2.0mm} \\
\lt(\beta^2 x - \beta, \frac 1\beta+\frac{y}{\beta^2}\rt), & {\rm \   if} \    (x, y) \in \lt[\frac 1\beta, \frac 1\beta + \frac{1}{\beta^2}\rt) \times [0, 1),

\vspace{2.0mm}

\\

\ \ \ \ \ \ \ldots \ ,

\\

\lt(\beta^mx - \sum_{j=1}^{m-1}\beta^{j}, \sum_{j=1}^{m-1}\beta^{-j}+ \frac{y}{\beta^m}\rt), & {\rm \   if} \  (x, y) \in \lt[\sum_{j=1}^{m-1}\beta^{-j}, 1\rt) \times [0, 1).
\end{cases}
\end{equation}

\

We will now recall the general concepts of \textit{first return time} and  \textit{first return map}.
\newline
 If $(X,\cF)$ is a $\sg$--algebra, $f:X \to X$ measurable, and $\mu$ a Borel probability $f$-invariant measure on $(X, \cF)$, let $A \subset X$ be a measurable set with $\mu(A) >0$. As we work with equilibrium measures, it is enough to take $A$ with $\Int (A)\ne \emptyset$.  By Poincar\'e Recurrence Theorem it  $\mu$-a.e $x \in A$ is recurrent, i.e it returns infinitely often to $A$ under iterates of $f$. Define then,
$$
n(x):= \inf\{m \ge 1, T^m(x) \in A\}
$$
This number being finite fro $\mu$--a.e. $x\in X$, is commonly called the \textit{ first return time} of $x$ to $A$. This permits to define the \textit{induced, or first return, map} $T_A:A \lra A$, by the formula
\begin{equation}\label{induce}
T_A(x) := T^{n(x)}(x).
\end{equation}
It is well known, and easy to prove,  that if 
$
\mu_A(B):= \frac{\mu(B)}{\mu(A)}
$,
for all $B \subset A$, then the probability measure $\mu_A$ on $A$ is $T_A$--invariant.
 
 \
 
It was proved  that the induced transformation of the natural extension $\mathcal T_\beta$ onto a certain subset, is isomorphic to the natural extension $\mathcal S$ of a GLS system (see for eg \cite{DK}). Recall that natural extensions are viewed only as dynamical systems, without measures. 
\newline
When $\beta=\frac{\sqrt 5 +1}{2}$, take the partition $\mathcal I = \{[0, \frac 1\beta), [\frac 1\beta, 1)\}$ and the associated GLS($\mathcal I$)-transformation
$$
S(x) :=  
\begin{cases}
          T_\beta(x), & {\rm if } \  x \in \lt[0, \frac 1\beta\rt) \\

        T_\beta^2(x), & {\rm  if } \ x \in [\frac 1\beta, 1) 
\end{cases}
$$
So if $\beta= \frac{\sqrt 5+1}{2}$, let  
$$
W := [0, 1) \times \lt[0, \frac 1\beta\rt).
$$
Then let
$$
\mathcal T_{\beta, W}: W \lra W,
$$
be the induced transformation of $\mathcal T_\beta$ on $W$.

If $(x, y) \in [0, \frac 1\beta) \times [0, 1)$, then
$$
\mathcal T_\beta(x, y) \in [0, 1) \times \lt[0, \frac 1\beta\rt),
$$
so for such $(x, y)$, we get $n(x, y) = 1$.  If $(x, y) \in [\frac 1\beta, 1) \times [0, \frac 1\beta)$, then $\mathcal T_\beta(x, y) \in [0, \frac 1\beta) \times [\frac 1\beta, 1) \notin W$, but $\mathcal T_{\beta}^2(x, y) \in W$, so $n(x, y) = 2$. \ 
Hence, the induced map $\mathcal T_{\beta, W}$ of the natural extension $\mathcal T_\beta$ on $W$ is,
\begin{equation}\label{induced-golden}
\mathcal T_{\beta, W}(x, y) 
=\begin{cases}
\lt(\beta x, \frac y\beta\rt),
&{\rm if} \ (x, y)  \in \lt[0, \frac 1\beta\rt)\times \lt[0, \frac 1\beta\rt) \\ \\
\lt(\beta(\beta x-1), \frac{y+ 1}{\beta^2}\rt) 
= \lt(\beta^2 x - \beta, \frac{y+1}{\beta^2}\rt),
&{\rm if} \ (x, y)  \in \lt[\frac 1\beta, 1\rt) \times \lt[0, \frac 1\beta\rt).
\end{cases}
\end{equation}
Then, from (\ref{natextGLS}), the inverse limit of $S$ is the map $\mathcal S:[0, 1)^2 \lra [0, 1)^2$ given by:
\begin{equation}\label{natextgolden}
\mathcal S(x, y) 
= \begin{cases}
\lt(\beta x, \frac y\beta\rt), &{\rm if}  \ (x, y)  \in \lt[0, \frac 1\beta\rt)\times [0, 1) \\  \\
\lt(\beta^2 x - \beta, \frac 1 \beta+ y \frac{\beta-1}{\beta}\rt) 
= \lt(\beta^2x-\beta, \frac{y+\beta}{\beta^2}\rt), &{\rm if} \ (x, y)   \in \lt[\frac 1\beta, 1) \times [0, 1\rt)\\
\end{cases}.
\end{equation}
If $\Psi:[0, 1)^2 \lra W$ is given by  
$$
\Psi(x, y) := \lt(x, \frac y \beta\rt),
$$
then $\Psi$ is an isomorphism between  $([0, 1)^2, \mathcal S)$ and $\lt([0, 1) \times [0, \frac 1\beta), \mathcal T_{\beta, W}\rt)$.
So $\mathcal T_{\beta, W}$ coincides (mod $\Psi$) with the natural extension $\mathcal S$ of GLS($\mathcal I$). 

Similarly, if $\beta>1$ satisfies 
$$
1 = \frac 1\beta + \ldots + \frac{1}{\beta^m}
$$ 
for some integer $m \ge 2$, then the associated natural extension  $\mathcal S$ from (\ref{mathcalSm}) is isomorphic to the induced transformation of $\mathcal T_\beta$ onto the rectangle $[0, 1) \times [0, \frac 1\beta)$.

\

For a \textit{general number} $\beta>1$, the situation is more complicated. 
\newline
First of all, \textit{not all sequences in $E_{[\beta]}^+$ are admissible}, i.e not all sequences of digits $(d_1, d_2, \ldots)$ determine a point $x \in [0, 1)$ that has $\beta$-expansion 
$$
x = \sum\limits_{n \ge 1} \frac{d_n}{\beta^n}
$$
This is an important obstacle, since we cannot code  $T_\beta$ with subshifts of finite type.
 For general $\beta>1$, one needs a more complicated GLS with partition $\mathcal I$ with countably many subintervals $I_n, n \in \mathcal D$, and then to induce the natural extension of $T_\beta$ on an appropriate subset, in order to obtain the natural extension of the GLS($\mathcal I$)-map. We will apply next Corollary \ref{c1_2016_12_16X} for equilibrium states on the natural extension of the GLS($\mathcal I$) expansion.
The construction of the inverse limit of $T_\beta$ can be found in \cite{DK}, \cite{DKS}, and we recall it here for the sake of completeness.  
Define the following rectangles 
$$
Z_0 := [0, 1)^2, \ \  \ Z_i:=[0, T_\beta^i1) \times \lt[0, \frac{1}{\beta^i}\rt), \ i \ge 1
$$
Consider the natural extension $Z$ obtained by placing each rectangle $Z_{i+1}$ on top of $Z_i$, for all $i \ge 0$. The index $i$ indicates at what height we are in this stack.  If $1$ has a finite $\beta$-expansion of length $n$, then only $n$ such rectangles $Z_i$ are stacked. 
 Assume $1$ has an infinite $\beta$--expansion; the finite case being treated in the same ways as  for $\beta = \frac{1+\sqrt 5}{2}$. 
 
Write the $\beta$--expansions of $1, x, y$, respectively as:
$$
1= .b_1b_2\ldots; \ \ \
x = .d_1d_2\ldots; \ \ \
y = .0\ldots 0 c_{i+1}c_{i+2}\ldots
$$
with $0$ repeated $i$ times. If $(x, y) \in Z_i$, then $d_1 \le b_{i+1}.$
Define $\mathcal T_\beta: Z \to Z, \ \mathcal T_\beta(x, y) = (T_\beta(x), \tilde y(x))$, with
\begin{equation}\label{generalb}
\tilde y(x) =  
\begin{cases}
\frac{b_1}{\beta}+\ldots + \frac{b_i}{\beta^i}+\frac{d_1}{\beta^{i+1}} + \frac{y}{\beta} = .b_1\ldots b_i d_1 c_{i+1}c_{i+2}\ldots, &{\rm if} \ d_1< b_{i+1}, \\
\frac y \beta,  &{\rm if} \  d_1 = b_{i+1}.
\end{cases}
\end{equation}
If $(x, y) \in Z_i$ then  $d_1 \le b_{i+1}$, so
\begin{equation}\label{bi}
\mathcal T_\beta(x, y) \in  
\begin{cases}
                   Z_0,  &{\rm if} \  d_1< b_{i+1}, \\
                  Z_{i+1},  &{\rm if} \   d_1 = b_{i+1}.
\end{cases}
\end{equation}
For $(x, y) \in Z_0$, if $d_1 < b_1$ then $\mathcal T_\beta(x, y) \in Z_0$; and if $d_i = b_i, 1 \le i \le n-1$ and $d_n < b_n$, then $\mathcal T_\beta^i(x, y) \in Z_i,  i \le n-1$ and $\mathcal T_\beta^n(x, y) \in Z_0$. Hence the induced map of $\mathcal T_\beta$ on $Z_0 = [0, 1)^2$ is 
$$
\mathcal T_{\beta, Z_0}(x, y) 
= \left\{ \begin{array}{ll}
                     \mathcal T_\beta(x, y), \  \ \, \text{if} \ d_1 < b_1, \\
\\
                  \mathcal T_\beta^n(x, y),  \  \  \text{if} \  d_i = b_{i}, 1 \le i \le n-1, \ \text{and} \ d_n < b_n\\
                     \end{array}
           \right.
$$
Partition $Z_0$ into subsets 
$$
Z_0^k:= \Big\{(x, y) \in Z_0, \ \inf\big\{n\ge 1, \mathcal T_\beta^n(x, y) \in Z_0\big\} = k\Big\}.
$$
Then,
\begin{equation}\label{natextTbetaZ0}
\mathcal T_{\beta, Z_0}(x, y) = \left\{ \begin{array}{ll}
                     \big(T_\beta(x), \frac 1\beta(y+d_1)\big), \ (x, y) \in Z_0^1,\\
\\
                  \big(T_\beta^k(x), \frac{b_1}{\beta} + \ldots + \frac{b_{k-1}}{\beta^{k-1}}+\frac{d_k}{\beta^k} + \frac{y}{\beta^k}\big), \ (x, y) \in Z_0^k, k \ge 2 \\
                     \end{array}
           \right.
\end{equation}

\begin{equation}\label{natextTbetaZ0}
\mathcal T_{\beta, Z_0}(x, y) = 
\begin{cases}
\lt(T_\beta(x), \frac 1\beta(y+d_1)\rt), &{\rm if} \ (x, y) \in Z_0^1,
\vspace{2.0mm} \\
\lt(T_\beta^k(x), \frac{b_1}{\beta} + \ldots + \frac{b_{k-1}}{\beta^{k-1}}+\frac{d_k}{\beta^k} + \frac{y}{\beta^k}\rt), &{\rm if} \ (x, y) \in Z_0^k, k \ge 2 \\
\end{cases}
\end{equation}
For any $n \ge 0$, if $b_0:= 0$, there exist unique integers $k = k(n) \ge 0$ and $1\le i \le b_{k+1}$ so that 
$$
n = b_0 + b_1 + \ldots + b_k +(i-1)
$$
Define a partition $\mathcal I = \{I_n, \ n \ge 0\}$ of $[0, 1)$ by
\begin{equation}\label{I_n}
I_n := \lt[b_0+\frac{b_1}{\beta} + \ldots + \frac{b_k}{\beta^k}+\frac{i-1}{\beta^{k+1}}, \ b_0+\frac{b_1}{\beta}+ \ldots + \frac{b_k}{\beta^k}+\frac{i}{\beta^{k+1}}\rt)
\end{equation}
From the definition of $\mathcal T_{\beta, Z_0}$ and of $I_n$, we see that for $(x, y) \in I_n\times [0, 1)$, we have:
$$
\mathcal T_{\beta, Z_0}(x, y) = \mathcal T_\beta^{k+1}(x, y) = \lt(T_\beta^{k+1}x, \ b_0+\frac{b_1}{\beta} + \ldots \frac{b_k}{\beta^k}+\frac{i-1}{\beta^{k+1}} + \frac{y}{\beta^{k+1}}\rt)
$$
If we take the transformation $S$ of GLS($\mathcal I$) and its natural extension $\mathcal S$, then (\ref{natextGLS}) applies.
If $x \in I_n$, then $s_1(x) = \beta^{k+1}$ and,
$$
h_1(x)/s_1(x) = b_0+\frac{b_1}{\beta} + \ldots + \frac{b_k}{\beta^k}+\frac{i-1}{\beta^{k+1}}
$$
Thus by (\ref{natextGLS}), $\mathcal S$ is equal to the induced map of the natural extension of $T_\beta$ on $Z_0$:
\begin{equation}\label{eq}
\mathcal S = \mathcal T_{\beta, Z_0}
\end{equation}
We can now apply (\ref{eq}) to equilibrium states of locally H\"older continuous potentials, for the induced map of the natural extension $\mathcal T_\beta$, in order to prove the \textit{exact dimensionality} of their conditional measures on fibers.
By (\ref{eq}) and Corollary \ref{c1_2016_12_16X}, we  obtain the following result,  for the induced map of the natural extension of the $\beta$--transformation:

\begin{thm}\label{golden}
Let $\beta >1$ arbitrary and let $T_\beta: [0, 1) \lra [0, 1)$ be the $\beta$-map, given by
$
T_\beta(x) = \beta x (\text{mod} \ 1).
$
Let $\phi: [0, 1)^2 \lra \R$ be a locally H\"older continuous map with 
$$
\sum\limits_{n \ge 1}\exp\(\sup\phi|_{I_n\times [0, 1)}\) < \infty,
$$
where $I_n$, $n \ge 0$, are given by (\ref{I_n}). 
Let $\mu_{\phi}$ be the equilibrium state of $\phi$ with respect to the induced map $\mathcal T_{\beta, [0, 1)^2}$ of the natural extension $\mathcal T_\beta$ on $[0, 1)^2$. Denote by $\mathcal S$ the natural extension of the GLS($\mathcal I$) map, where $\mathcal I$ is the partition of $[0, 1)$ given by $(I_n)_{n \ge 0}$. 

\fr Then for $\mu_\phi\circ\pi_0^{-1}$--a.e $x \in [0, 1)$,  conditional measure $\mu_\phi^x$ is exact dimensional on $[0, 1)$ and 
$$
\HD(\mu_\phi^x) = \lim\limits_{r\to 0} \frac{\mu_\phi^x(B(y, r))}{\log r} = \frac{\h_{\mu_{\phi}}(\mathcal S)}{\chi_{\mu_{\phi}}(\mathcal S)}
$$
\end{thm}
 
\

Due to the expression of the induced map on $[0, 1)^2$ of the natural extension $\mathcal T_\beta$, as being a natural extension for a GLS-map, we can say more about the \textit{Lyapunov exponent} $\chi_{\mu_\phi}(\mathcal S)$ and the dimension of conditional measures:

 \begin{cor}\label{Lyapunov-GLS} In the setting of Theorem \ref{golden} with $\beta>1$ arbitrary, write the $\beta$-expansion of 1 as $1=.b_1b_2\ldots$. For an arbitrary integer $n \ge 0$, define the integers $k = k(n)\ge 0$ and $1\le i = i(n)\le b_{k+1}$ so that $n = b_1+\ldots +b_k + i-1$.
 Then, 
 \begin{itemize} 
\item[(a)] 
with the intervals $I_n, n \ge 0$ given by (\ref{I_n}), we obtain the Lyapunov exponent as,  
$$
\chi_{\mu_\phi}(\mathcal S) = \log \beta \cdot \mathop{\sum}\limits_{n \ge 0} \big(k(n)+1\big)\cdot \mu_\phi\big(I_n \times [0, 1)\big)
$$
Hence for $\mu_\phi\circ\pi_0^{-1}$--a.e $x \in [0, 1)$, we have: \ 
$
\HD(\mu_\phi^x) = \frac{\h_{\mu_{\phi}}(\mathcal S)}{\log \beta  \cdot \mathop{\sum}\limits_{n \ge 0} (k(n)+1)\cdot \mu_\phi\big(I_n \times [0, 1)\big)}.
$
\item [(b)] When $\beta = \frac{1+\sqrt 5}{2}$, we obtain that the Lyapunov exponent of $\mu_\phi$ is equal to
$$
 \chi_{\mu_\phi}(\mathcal S) 
 = \log \frac{1+\sqrt 5}{2} \cdot \lt(1 + \mu_\phi\Big(\Big[\frac 1\beta, 1\Big) \times [0, 1)\Big)\rt).
$$
Hence, for $\mu_\phi\circ\pi_0^{-1}$--a.e $x \in [0, 1)$, we have: \ 
$
\HD(\mu_\phi^x) = \frac{\h_{\mu_{\phi}}(\mathcal S)}{\log \frac{1+\sqrt 5}{2}  \Big(1 + \mu_\phi\big([\frac 1\beta, 1) \times [0, 1)\big)\Big)}.
$

\item [c)] Let $\beta>1$ be the positive root of the polynomial $z^m - z^{m-1} - \ldots -1$, so that $1 = \frac 1\beta + \frac{1}{\beta^2} + \ldots + \frac{1}{\beta^m}$. Let $\mathcal S$ from (\ref{mathcalSm}) be the natural extension of the associated GLS.  Then, the Lyapunov exponent of $\mu_\phi$ is equal to 
$$
\chi_{\mu_\phi}(\mathcal S) = \log \beta \cdot \lt(1+ \mu_\phi\Big([\frac 1\beta, \frac 1\beta+\frac{1}{\beta^2})\Big) + \ldots + (m-1) \mu_\phi\Big([\frac{1}{\beta} + \ldots + \frac{1}{\beta^{m-1}}, 1)\Big)\rt)$$
Hence for $\mu_\phi\circ\pi_0^{-1}$--a.e $x \in [0, 1)$, 
$$
\HD(\mu_\phi^x) = \frac{\h_{\mu_{\phi}}(\mathcal S)}{\log \beta \cdot \lt(1+ \mu_\phi\Big([\frac 1\beta, \frac 1\beta+\frac{1}{\beta^2})\Big) + \ldots + (m-1) \mu_\phi\Big([\frac{1}{\beta} + \ldots + \frac{1}{\beta^{m-1}}, 1)\Big)\rt)}$$
\end{itemize}
\end{cor}
 
\begin{proof}
In order to prove (a) we apply (\ref{I_n}), and (\ref{eq}). Let us write 
$$
\mathcal S(x, y) = \(S(x), g_x(y)\), \ (x, y) \in [0, 1)\times [0, 1),
$$
where $S$ is the GLS($\mathcal I$)-transformation and $\mathcal S$ is its natural extension (see \ref{natextGLS}). The derivative of the fiber map $g_x$ is constant and equal to $L_n$, for $x \in I_n$, where $L_n$ is the length of the interval $I_n$ and thus is equal to $\frac{1}{\beta^{k(n)+1}}$, for $n \ge 0$. Finally, for the Hausdorff (and pointwise) dimension of conditional measures we apply Theorem \ref{golden}.
 
For (b) we apply (\ref{natextgolden}) to get, for $\beta = \frac{1+\sqrt 5}{2}$,
$$
g_x'(y)= \frac 1\beta, \ \ 
\text{if} \ (x, y) \in [0, \frac 1 \beta)\times [0, 1), \ \text{and}, \ 
g_x'(y) = \frac{1}{\beta^2}, \ 
\text{if} \ (x, y) \in [\frac 1\beta, 1) \times [0, 1)$$
Therefore in this case, $k(0) = 0$ and $k(1) = 1$, which we use in the formula for the Lyapunov exponent.
Then we use the fact that  
$$
\mu_\phi\lt(\Big[0, \frac 1 \beta\Big)\times [0, 1)\rt) + \mu_\phi\lt(\Big[\frac 1\beta, 1\Big) \times [0, 1\rt) = 1.
$$

For (c), since $b_1 = \ldots = b_m = 1$, it follows that $$k(1) = 1, k(2) = 2, \ldots, k(m) = m$$ Then we use the definition of the Lyapunov exponent 
$$
\chi_{\mu_\phi}(\mathcal S) = - \int_{[0, 1)^2} \log|g_x'(y)| d\mu_\phi(x,y),
$$
and the fact that  
$$
\mu_\phi\lt(\lt[0, \frac 1\beta\rt)\rt) + \mu_\phi\lt(\lt[\frac 1\beta, \frac 1\beta+\frac{1}{\beta^2}\rt)\rt) + \ldots + \mu_\phi\lt(\lt[\frac{1}{\beta} + \ldots + \frac{1}{\beta^{m-1}}, 1\rt)\rt) = 1.
$$
\end{proof}

From (\ref{eq}), we know that the induced map $\mathcal T_{\beta, [0, 1)^2}$ is equal to the inverse limit $\mathcal S$ of the GLS transformation $S$ associated to the countable partition $\mathcal I$, given by (\ref{I_n}). From (\ref{natextGLS}) and as $f$ is given by (\ref{A1_2017_03_16}), then $\mathcal S(x, y)$ and $\mathcal T_{\beta, [0, 1)^2}$ satisfy
\begin{equation}\label{ultimaeq}
\mathcal S(x, y)
 = \mathcal T_{\beta, [0, 1)^2}(x, y) = \lt(f(x), \frac{h_1}{s_1} + \frac{y}{s_1}\rt), \ (x, y) \in [0, 1)^2
\end{equation}

\

We use now the explicit form of $\mathcal T_{\beta, W}$ and Theorem~\ref{globalexact}, to show that any $\mathcal T_{\beta, W}$--equilibrium measure $\mu_\phi$ 
 is exact dimensional on $[0, 1)^2$, and to compute its dimension.
Notation is of Corollary~\ref{Lyapunov-GLS}; and $\pi_1: [0, 1)^2 \to [0, 1)$ is the projection on first coordinate.

\

\begin{thm}\label{beta-global}

Let an arbitrary $\beta >1$, $T_\beta(x) = \beta x \ ({\rm mod} \, 1)$, $x \in [0, 1)$, and let $\mathcal T_\beta$ be the natural extension of $T_\beta$, and $\mathcal S = \mathcal T_{\beta, [0, 1)^2}$ be the induced map of $\mathcal T_\beta$ on $[0, 1)^2$. Recall  the associated map $f$ from (\ref{A1_2017_03_16}), (\ref{ultimaeq}).
Let  $\phi:[0, 1)^2\lra\R$ be a locally H\"older continuous potential which satisfies,
$$
\sum\limits_{n \ge 1}\exp\(\sup\phi|_{I_n\times [0, 1)} \) < \infty,
$$
where the subintervals $I_n, n \ge 0$ are given by (\ref{I_n}). 
Denote by $\mu_\phi$ the unique equilibrium measure of $\phi$ with respect to $\mathcal T_{\beta, [0, 1)^2}$. Define
$
\nu:=\mu_\phi\circ\pi_{1}^{-1}
$
as the projection of $\mu_\phi$ on the first coordinate. 
Then, $\mu_\phi$ is exact dimensional on $[0, 1)^2$ and,
$$
\HD(\mu_\phi) = \frac{2\h_{\nu}(f)}{\log \beta \cdot \mathop{\sum}\limits_{n \ge 0} \big(k(n)+1\big)\mu_\phi\big(I_n \times [0, 1)\big)}
$$
\end{thm}

\begin{proof}

First we will show that $\mu_\phi$ is exact dimensional. Recall from (\ref{ultimaeq}) that for $\beta>1$ arbitrary, and for all $(x, y) \in [0, 1)^2$, $$\mathcal T_{\beta, [0, 1)^2}(x, y)  = \mathcal S(x, y) = \lt(f(x), \frac{h_1}{s_1} + \frac{y}{s_1}\rt)$$ Note that
the projection $\nu$ is $f$--invariant and ergodic on $[0, 1)$.

Now, if $\mathcal I$ denotes the countable partition $(I_n)_{n \ge 1}$ from (\ref{I_n}), then from the previous Section we obtain the conjugacy 
$$
\pi:\Sigma_{\mathcal I}^+ \lra [0, 1)
$$ 
between $(\Sigma_{\mathcal I}^+, \sigma)$ and $([0, 1), f)$. Moreover,  we also have that $\nu$ gives zero measure to points, and thus  $\nu$ is the projection of an ergodic measure $\tilde \nu$ on $\Sigma_{\mathcal I}^+$. The interval $[0, 1)$ is viewed as the limit set of the iterated system associated to the countable partition $\mathcal I$ of $[0, 1)$, where the contractions are the inverses of the branches of $f$ on $I_n, n \ge 1$.

 Consider now, in the notation of \cite{MU-Adv}, the random system given by a parameter space $\Lambda = \{\lambda\}$ with the identity homeomorphism $\theta: \Lambda \to \Lambda, \ \theta(\lambda) = \lambda,$ which preserves the Dirac delta measure $\delta_\lambda$, and the shift space $(\Sigma_{\mathcal I}^+, \sigma)$ with the ergodic $\sigma$--invariant measure $\tilde\nu$.
Then, the measure $\tilde\nu$ is in fact the only conditional measure of the product measure $\delta_\lambda \times \tilde\nu$ on $\Lambda \times \Sigma_{\mathcal I}^+$. 

On the other hand, since the potential $\phi$ is summable, it follows from our results in  \ref{con}  and from the fact that $\big([0, 1)^2, \mathcal T_{\beta, [0, 1)^2}\big)$ is coded by a Smale space of countable type, that the entropy $\h_{\mu_\phi}(\mathcal T_{\beta, [0, 1)^2})$ is finite. Therefore, since $$\nu = \mu_\phi\circ\pi_{1}^{-1},$$ we obtain that  $\h_{\tilde\nu}(f)< \infty$. But then we get from Remark~3.4 of \cite{MU-Adv} that 
$$
\H_{\delta_\lambda\times \tilde \nu}\(\pi^{-1}_{\Sigma_I^+}(\xi)|\pi_\Lambda^{-1}(\epsilon_\Lambda)\) < \infty
$$
Hence, it follows from Theorem~3.13 in \cite{MU-Adv} and the discussion above,  that $\nu$ is exact dimensional on $[0, 1)$, and that 
$$
\HD(\nu) = \frac{\h_{\nu}(f)}{\chi_\nu(f)}
$$

On the other hand, it follows from Theorem \ref{golden} that the conditional measures of $\mu_\phi$ are exact dimensional on fibers (which fibers are all equal to $[0, 1)$ in this case).

 Now, if the $f$--invariant measure $\nu$ is exact dimensional on $[0, 1)$, and if the conditional measures of $\mu_\phi$ are exact dimensional on fibers,  we apply Theorem~\ref{globalexact}  to obtain that $\mu_\phi$ is exact dimensional globally on $[0, 1)^2$. From Theorem \ref{globalexact},  $\HD(\mu_\phi)$ is the sum of $\HD(\nu)$ and of the dimension of conditional measures. 

The last step is then to obtain the expression of the Lyapunov exponent $\chi_\nu(f)$. In our case, it follows from (\ref{A1_2017_03_16}) that  $|f'(x)|=L_n^{-1}$ for all $x \in I_n$, where $L_n$ denotes the length of $I_n$, for $n \ge 0$. In addition, it follows from (\ref{I_n}) that 
$$
L_n = \frac{1}{\beta^{k(n)+1}}
$$
But also 
$
\nu(A) = \mu_\phi(A \times[0, 1)),
$
for any Borel set $A \subset [0, 1)$. Therefore, 
$$
\chi_\nu(f) = \log \beta \cdot \mathop{\sum}\limits_{n \ge 0}\big(k(n)+1\big) \mu_\phi\big(I_n \times [0, 1)\big)
$$
In addition  by the Shannon-McMillan-Breiman Theorem (\cite{Man}), $\h_{\mu_\phi}(\mathcal S)$ is equal to $\h_\nu(f)$ since $\mathcal S$ contracts in the second coordinate.  
Then, from the above discussion and Corollary~\ref{Lyapunov-GLS},  the Hausdorff dimension of $\mu_\phi$ is obtained as in the statement.

\end{proof}

\

\textbf{Acknowledgements:}  E. Mihailescu thanks Institut des Hautes \'Etudes Sci\'entifiques, Bures-sur-Yvette, France, for a research stay when part of this paper was written. The research of the second named author was funded in part by the Simons Foundation 581668.

\end{document}